\documentclass{amsart}
\usepackage{amssymb,amsfonts, color,epsf}
\usepackage [cmtip,arrow]{xy}
\xyoption{all}
\usepackage {pb-diagram,pb-xy}
\usepackage{enumerate}
\usepackage{accents}
\usepackage[noautoscale]{youngtab}
\usepackage{subfigure}
\ifx\pdfpageheight\undefined
   \usepackage[dvips,colorlinks=true,linkcolor=blue,citecolor=red,%
      urlcolor=green]{hyperref}
   \usepackage[dvips]{graphicx}
   \makeatletter
   \edef\Gin@extensions{\Gin@extensions,.mps}
   \makeatother
\else
 \usepackage[pdftex]{graphicx}
   \usepackage[bookmarksopen=false,pdftex=true,breaklinks=true,%
      backref=page,pagebackref=true,plainpages=false,%
      hyperindex=true,pdfstartview=FitH,colorlinks=true,%
      pdfpagelabels=true,colorlinks=true,linkcolor=blue,%
      citecolor=red,urlcolor=green,hypertexnames=false%
      ]%
   {hyperref}
\fi
\usepackage{bm}

\newcommand*{\medcap}{\mathbin{\scalebox{1.5}{\ensuremath{\cap}}}}

\usepackage{amsmath,amssymb,amsfonts}
\newtheorem{theorem}{Theorem}
\newtheorem{lemma}{Lemma}
\newtheorem{corollary}{Corollary}
\newtheorem{proposition}{Proposition}

\theoremstyle{definition}
\newtheorem{definition}{Definition}
\newtheorem{example}{Example}

\newtheorem{notation}{Notation}

\theoremstyle{remark}
\newtheorem{remark}{Remark}

\definecolor{DarkBlue}{rgb}{0,0.1,0.55}

\numberwithin{equation}{section}

\newcommand {\hide}[1]{}
 \newcommand {\sign} {\mbox{\bf sign}}

\newcommand {\junk}[1]{}

\newcommand {\R} {\mathrm{R}}

\newcommand {\D}     {\mbox{\rm D}}

\newcommand {\C}     {\mathrm{C}}

\newcommand {\Z}  {\mathbb{Z}}
 
\newcommand {\Q}         {\mathbb{Q}}
\newcommand {\kk}         {\mathbf{k}}

\newcommand{\dd} {\mathbf{d}}
\newcommand{\Qu} {\mathrm{Q}}
\newcommand{\F}{\mathbb{F}}

\newcommand {\ZZ} {{\rm Z}}
\newcommand {\RR} {{\mathcal R}}

\newcommand {\la}   {{\langle}}
\newcommand {\ra}   {{\rangle}}
\newcommand {\eps} {{\varepsilon}}

\newcommand{\card}{\mathrm{card}}

\def\addots{\mathinner{\mkern1mu
\raise1pt\vbox{\kern7pt\hbox{.}}
\mkern2mu\raise4pt\hbox{.}\mkern2mu
\raise7pt\hbox{.}\mkern1mu}}

\newcommand{\HH}  {\mbox{\rm H}}

\newcommand{\x}{\mathbf{x}}
\newcommand{\X}{\mathbf{X}}

\newcommand{\y}{\mathbf{y}}
\newcommand{\Y}{\mathbf{Y}}

\newcommand{\ZB}{\mathbf{Z}}

\newcommand{\Ext}{\mathrm{Ext}}
\newcommand{\length}{\mathrm{length}}

\newcommand{\Comp}{\mathrm{Comp}}
\newcommand{\CompMax}{\mathrm{CompMax}}

\newcommand{\Weyl}{\mathcal{W}}
\newcommand{\w}{\mathbf{w}}
\newcommand{\boldPsi}{\boldsymbol{\Psi}}
\newcommand{\boldpi}{\boldsymbol{\pi}}
\newcommand{\boldsigma}{\boldsymbol{\sigma}}
\newcommand{\boldSigma}{\boldsymbol{\Sigma}}

\newcommand{\bWeyl}{\boldsymbol{\mathcal{W}}}
\newcommand{\bComp}{\mathbf{Comp}}
\newcommand{\boldlambda}{\boldsymbol{\lambda}}
\newcommand{\boldLambda}{\boldsymbol{\Lambda}}
\newcommand{\boldmu}{\boldsymbol{\mu}}
\newcommand{\boldnu}{\boldsymbol{\nu}}

\begin{document}
\title[On the equivariant Betti numbers of  symmetric 
definable
sets]
{
On the equivariant Betti numbers of  symmetric 
definable
sets: vanishing, bounds and algorithms
}
\author{Saugata Basu}
\address{Department of Mathematics,
Purdue University, West Lafayette, IN 47906, U.S.A.}
\email{sbasu@math.purdue.edu}

\author{Cordian Riener}
\address{Department of Mathematics and Statistics, 
UiT The Arctic University of Norway,  
9037 Troms\o{}, Norway}
\email{cordian.riener@uit.no}

\date{\textbf{\today}}
\thanks{
Basu was also  partially supported by NSF grants
CCF-1618918 and DMS-1620271. 
 }

\begin{abstract}
Let $\R$ be a real closed field.
We prove that for any fixed $d$, the equivariant rational cohomology groups of closed symmetric semi-algebraic subsets of $\R^k$ defined
by polynomials of degrees bounded by $d$ vanishes in dimensions $d$ and larger.  
This vanishing result is tight.
Using a new geometric approach we also prove an upper bound of $d^{O(d)} s^d k^{\lfloor d/2 \rfloor-1} $ on the equivariant Betti numbers of closed symmetric semi-algebraic subsets of $\R^k$ defined by quantifier-free formulas  involving $s$ symmetric polynomials of degrees bounded by $d$, where $1 < d \ll s,k$.
This bound is tight up to a factor depending only on $d$.
These results significantly improve upon  those obtained previously in \cite{BC2013} which were proved using different techniques.
Our new methods are quite general, and also yield bounds on the equivariant Betti numbers of certain special classes of symmetric definable sets (definable sets symmetrized by pulling back under symmetric polynomial maps of fixed degree)  in arbitrary  o-minimal structures over $\R$.

Finally, we utilize our new approach to obtain an algorithm with polynomially bounded complexity for computing these equivariant Betti numbers. In contrast, the problem of computing the ordinary Betti numbers of (not necessarily
symmetric) semi-algebraic sets is considered  to be an intractable problem, and all known algorithms for this problem have doubly exponential complexity.
\end{abstract}
\maketitle
\tableofcontents

\section{Introduction}
\label{sec:intro}
The problem of bounding the Betti numbers of  semi-algebraic sets defined over
the real numbers has a long history, and has attracted the attention of many
researchers -- starting from the first results due to 
Ole{\u\i}nik and Petrovski{\u\i}
{\cite{OP}}, followed by Thom {\cite{T}}, Milnor {\cite{Milnor2}}. 
If there is an action of a (compact) group on a real vector space whose action leaves the given semi-algebraic set invariant, it makes sense to separately study the topology modulo the group action. One classical notion to do this is by means of the so called \emph{equivariant} Betti numbers (see \S  \ref{def:equivariant-cohomology}).  The resulting  question of studying the equivariant Betti numbers of $\emph{symmetric}$ semi-algebraic subsets of $\R^k$ is relatively more recent and was initiated in \cite{BC2013}, where polynomial bounds for semi-algebraic sets defined by symmetric polynomials were given. 

Before proceeding any further it will be useful to keep in mind the 
following simple example (both as a guiding principle for proving upper bounds on and as a lower bound for the equivariant Betti numbers).

\begin{example}
\label{eg:key}
Let $1 < d \ll k$, $d$ even. We will think of $d$ as a fixed constant and let $k$ be large.
Also, let
\[
P = \sum_{i=1}^{k} \prod_{j=1}^{d/2} (X_i - j)^2 \in \R[X_1,\ldots,X_k].
\]
Then, the set of real zeros, $V_{d,k}$ of $P$ in $\R^k$ is finite and consists of the $(d/2)^k$ isolated points -- namely the set $\{1,\ldots,d/2\}^k$. In other words the zero-th Betti number of $V_{d,k}$ equals 
\[
   (d/2)^k = (O(d))^k, 
\]
which grows exponentially  in $k$ (for fixed $d$). However, $P$ is a symmetric polynomial, and as a result there is an action of the symmetric group $\mathfrak{S}_k$ on $V_{d,k}$.
The number of orbits of this action equals the zero-th Betti number of the quotient $V_{d,k}/\mathfrak{S}_k$. It is not too difficult to see that the orbit of a point $\x=(x_1,\ldots,x_k) \in V_{d,k}$ is determined by the tuple $\lambda(\x) = (\lambda_1,\ldots,\lambda_{d/2})$, where
$\lambda_i = \card(\{j \mid x_j = i\})$. Thus, the number of orbits of $V_{d,k}$, and thus the sum of the Betti numbers of the quotient $V_{d,k}/\mathfrak{S}_k$ equals
$\binom{k + d/2 -1}{d/2-1}$, which satisfies the inequalities
\begin{equation*}
\label{eqn:eg:key:1}
c_d \cdot k^{d/2-1} \leq \binom{k + d/2 -1}{d/2-1}  \leq C_d \cdot k^{d/2-1},
\end{equation*}
where $c_d, C_d$ are constants that depend only on $d$. Notice that unlike the Betti numbers of $V_{d,k}$ itself,
the Betti numbers of the quotient are bounded by a \emph{polynomial} in $k$ 
(for fixed $d$), and moreover the
degree of this polynomial is $d/2-1$. One of the main new results of the current paper
(see inequality \eqref{eqn:thm:bound:1} in Theorem \ref{thm:bound})
is an upper bound on the sum of the equivariant Betti numbers of symmetric real varieties 
that matches (up to a factor depending only on $d$) the lower bound
implied by Example \ref{eg:key}. 
\end{example}

In the present article we improve the existing quantitative results on the vanishing of the higher equivariant cohomology groups of symmetric semi-algebraic sets 
(Theorem \ref{thm:vanishing})
as well as bounding of the equivariant Betti numbers of symmetric semi-algebraic sets 
(Theorems \ref{thm:bound} and \ref{thm:bound-general}).
Our techniques are completely different than those used in \cite{BC2013} where the previous best known bounds for these quantities were proved. 
Moreover, the new methods also yield bounds on the equivariant Betti numbers of certain special classes of symmetric definable sets (definable sets symmetrized by pulling back under symmetric polynomial maps of fixed degree)  in arbitrary  o-minimal structures over $\R$ (Theorems \ref{thm:definable1} and \ref{thm:definable2}).

While obtaining tight upper bounds on the Betti numbers of real varieties and semi-algebraic sets is an extremely well-studied problem \cite{BPR10}, there is also a related \emph{algorithmic} question that is of great importance -- namely, designing efficient algorithms for computing them. 
One reason for the importance of this algorithmic question is that the existence or non-existence 
of such algorithms with \emph{polynomially bounded complexity} for real varieties defined by polynomials of degrees bounded by some constant is closely related to the $\mathbf{P}_{\mathbb{R}}$ versus 
$\mathbf{NP}_{\mathbb{R}}$ and similar questions in the Blum-Shub-Smale theory of computation and its generalizations (see for example \cite{BSS89}).

The  new method used in the proof for the tighter bounds allows us to give an algorithm with polynomially bounded complexity for computing the equivariant Betti numbers of semi-algebraic sets defined by symmetric polynomials of degrees bounded by some constant (Theorem \ref{thm:algorithm}). 
This is particularly striking because the problem of computing the \emph{ordinary} Betti numbers in the non-symmetric case is a 
$\mathbf{PSPACE}$-hard problem, and is thus considered intractable.
In particular, this result also confirms a meta-theorem that suggests that for computing polynomially bounded topological invariants of semi-algebraic sets algorithms with polynomially bounded complexity should exist. 

\subsection{Notations and background}
All our results will be stated not only for the real numbers but more generally for arbitrary real closed fields. Note however, that by the 
\emph{Tarski-Seidenberg transfer theorem} (the reader may consult \cite[Chapter 2]{BPRbook2} for a detailed exposition of this statement)  most statements valid over one such field hold in  any other real closed field. Therefore, we can fix a real closed field $\R$, and we denote by
$\C$  the algebraic closure of $\R$.   We also introduce the following notation.

\begin{notation}
Given $k,d \in \Z_{\geq 0}$, we denote by $\R[\X]_{\leq d}= \R[X_1,\ldots,X_k]_{\leq d}$ the $\R$-vector space of polynomials of
degree at most $d$. More generally, given $\kk = (k_1,\ldots,k_\omega),\dd = (d_1,\ldots,d_\omega) \in \Z_{\geq 0}^\omega$, we will denote
\[
\R[\X^{(1)},\ldots,\X^{(k_\omega)}]_{\leq \dd} = \R[\X^{(1)}]_{\leq d_1} \otimes \cdots \otimes \R[\X^{(\omega)}]_{\leq d_\omega},
\]
where for each $i, 1 \leq i \leq \omega$,
\[
\R[\X^{(i)}] = \R[X_1^{(i)},\ldots,X_{k_i}^{(i)}].
\]
For $\kk =(k_1,\ldots,k_\omega) \in \Z_{\geq 0}^\omega$, we will also denote by $|\kk| = \sum_{i=1}^\omega k_i$.
\end{notation}

\begin{notation}
  For a given polynomial  $P \in \R [X_{1} , \ldots ,X_{k} ]$ we denote the set of zeros of $P$  in  $\R^{k}$ by $\ZZ (P, \R^{k})$. More generally, for any finite set
  $\mathcal{P} \subset \R [ X_{1} , \ldots ,X_{k} ]$, the set of common zeros of $\mathcal{P}$ in
  $\R^{k}$ is denoted by $\ZZ
  (\mathcal{P}, \R^{k})$.
\end{notation}

\begin{definition}
  \label{def:sign-condition}
  Let  $\mathcal{P}\subset \R [ X_{1} , \ldots ,X_{k} ]$ be a finite family of polynomials. 
  An element  $\sigma \in \{0,1,-1 \}^{\mathcal{P}}$ is called  a \emph{sign condition} on $\mathcal{P}$.
  Given any semi-algebraic set $Z \subset \R^{k}$, and  a sign condition $\sigma \in
  \{ 0,1,-1 \}^{\mathcal{P}}$, the \emph{realization} of $\sigma$ on $Z$ is the 
   semi-algebraic set defined by 
  \[
  \left\{ \x \in Z \mid \sign (P (\x)) = \sigma (P)  ,P \in \mathcal{P} \right\}.
  \]
  More generally,  let  $\Phi$ be a Boolean formula such that the  atoms of $\Phi$ are of the from, 
  $P \;\sim\; 0, P \in \mathcal{P}$, where the relation $\sim$ is one of $=,>,$ or $<$. Then we  will call such a formula a
   \emph{$\mathcal{P}$-formula}. and  the \emph{realization} of $\Phi$, i.e.,  the semi-algebraic set
  \begin{eqnarray*}
    \RR (\Phi , \R^{k}) & = & \{ \x \in \R^{k} \mid
    \Phi (\x) \},
  \end{eqnarray*}  will be called  a \emph{$\mathcal{P}$-semi-algebraic set}. 
  Finally, a  Boolean formula without negations, and with atoms 
  $P \;\sim\; 0, P\in \mathcal{P}$ where $\sim$ is either $\leq$ or $\geq$, 
  will be called a \emph{$\mathcal{P}$-closed formula}, and we call
  its realization, 
  $\RR(\Phi , \R^{k})$, a \emph{$\mathcal{P}$-closed semi-algebraic set}.
\end{definition}

\begin{notation}
  Let $X\subset\R^{k}$ be any  semi-algebraic set and let $\F$ be a fixed field.
Then, we will consider the  \emph{$i$-th cohomology group} of $X$ with coefficients in
 $\F$, which is denoted  by $\HH^{i} (X,\F)$. We will study the dimensions of these $\F$ vector spaces, which are denoted by $b^{i} (X,\F) =  \dim_{\F}  \HH^{i} (X,\F)$, and their sum denoted by $b(X,\F) = \sum_{i \geq 0}b^i(X,\F)$. 
 It is worth noting  that the precise  definition of  these notions  requires some care if the semi-algebraic set is defined over an arbitrary (possibly non-archimedean) real closed field. For details we refer to \cite[Chapter 6]{BPRbook2}. \end{notation}
 
The following classical result, which is due to Ole{\u\i}nik and Petrovski{\u\i} \cite{OP},
Thom \cite{T},  and Milnor \cite{Milnor2} gives a sharp upper bound on the Betti
numbers of a real variety in terms of the degree of the defining polynomial
and the number of variables.

\begin{theorem}
  \label{thm:betti-bound-algebraic}{\cite{OP,T,Milnor2}} 
  Let $k,d \in \Z_{\geq 0}$, and 
  $Q  \in \R [ X_{1} , \ldots ,X_{k} ]_{\leq d}$.
  Then, for any field of coefficients $\F,$
  \begin{eqnarray*}
    b(\ZZ (Q, \R^{k}) ,\F) & \leq & d (2d-1)^{k-1} .
  \end{eqnarray*}
\end{theorem}

More generally for $\mathcal{P}$-closed semi-algebraic sets we have the following bound. 

\begin{theorem}\cite{BPRbook2,GV07}
  \label{thm:betti-bound-sa-general} 
  Let $k,d \in \Z_{\geq 0}$,
  $\mathcal{P} \subset \R [ X_{1} , \ldots ,X_{k} ]_{\leq d}$ be a finite set of
  polynomials, and $S$ be a
  $\mathcal{P}$-closed semi-algebraic set. Then, for any field of coefficients
  $\F,$
  \begin{eqnarray*}
    b(S,\F) & \leq & \sum_{i=0}^{k}
    \sum_{j=1}^{k-i} \binom{\card(\mathcal{P})+1}{j} 6^{j} d (2d-1)^{k-1} .
  \end{eqnarray*}
 \end{theorem}

We will need the following immediate corollary of Theorem \ref{thm:betti-bound-sa-general}. Using the same notation as in 
Theorem \ref{thm:betti-bound-sa-general} we have:
\begin{corollary}
\label{cor:betti-bound-sa-general}
Suppose that $L \subset \R^k$ is a subspace with $\dim L = k'$. Then, for any field of coefficients
  $\F,$
  \begin{eqnarray*}
    b(L \cap S,\F) & \leq & \sum_{i=0}^{k'}
    \sum_{j=1}^{k'-i} \binom{\card(\mathcal{P})+1}{j} 6^{j} d (2d-1)^{k'-1} .
  \end{eqnarray*}
\end{corollary}

\begin{proof}
Note that a polynomial of degree bounded by $d$ in $\R^k$, pulls back to a polynomial on $L$ of degree at most $d$,
under the inclusion $\iota:L \hookrightarrow \R^k$. The corollary now follows immediately from Theorem
\ref{thm:betti-bound-sa-general}.
\end{proof}

In this paper we will consider bounding the \emph{equivariant} Betti numbers of symmetric semi-algebraic sets in terms of the 
\emph{multi-degrees}
of the defining polynomials. For this purpose it will be useful to have a more refined bound than the one in Theorem
\ref{thm:betti-bound-sa-general}. The following bound appears in \cite{BR2015}. Notice that in contrast to 
Theorems \ref{thm:betti-bound-sa-general} and
\ref{thm:betti-bound-algebraic} above which holds for coefficients in an arbitrary field $\F$, Theorem \ref{thm:multi-semi} only provides bounds
for the $\Z_2$-Betti numbers only. However, using the universal coefficients theorem, it is clear that a bound on the $\Z_2$-Betti is also a bound on the rational Betti numbers.

\begin{theorem}
\label{thm:multi-semi}
Let $\kk =(k_1,\ldots,k_\omega),\dd=(d_1,\ldots,d_\omega)  \in \Z_{\geq 0}^\omega$, 
$k = |\kk|$, $d_i \geq 2, 1 \leq i \leq \omega$, and 
$\mathcal{P} 
\subset \R[\X^{(1)},\ldots,\X^{(p)}]_{\leq \dd}$ 
a finite set of polynomials, where for $1 \leq i \leq \omega$, 
$\X^{(i)} = (X^{(i)}_1,\ldots,X^{(i)}_{k_i})$. 

If $S$ is a $\mathcal{P}$-closed semi-algebraic set, then
\begin{eqnarray*}
b(S,\Z_2) 
&\leq& O(1)^k \card(\mathcal{P})^k \omega^{3k} d_1^{k_1}\cdots d_\omega^{k_\omega}.
\end{eqnarray*}
\end{theorem}

We will need the following immediate corollary of Theorem \ref{thm:multi-semi}. Using the same notation as in 
Theorem \ref{thm:multi-semi} we have:
\begin{corollary}
\label{cor:multi-semi}
Suppose for $1 \leq i \leq \omega$, $L^{(i)}  \subset \R^{k_i}$ is a subspace with $\dim L^{(i)} = k_i'$, and 
$L = \oplus_{i=1}^{\omega} L^{(i)}$, and $k' = \sum_{i=1}^{\omega} k_i'$.  Then, 
\begin{eqnarray*}
b(S \cap L,\Z_2) 
&\leq& O(1)^{k'} \card(\mathcal{P})^{k'} \omega^{3k'} d_1^{k_1'}\cdots d_\omega^{k_\omega'}.
\end{eqnarray*}
 \end{corollary}

\begin{proof}
Note that a polynomial of multi-degree bounded by $\dd$ in $\R^{k_1} \times \cdots \times \R^{k_\omega}$, pulls back to a polynomial on $L$ of degree at most $\dd$,
under the inclusion 
\[
\iota=(\iota_1\oplus \cdots\oplus \iota_\omega):L^{(1)} \oplus \cdots \oplus L^{(\omega)} \hookrightarrow \R^{k_1} \oplus \cdots \oplus \R^{k_\omega}.
\] 
The corollary now follows immediately from Theorem
\ref{thm:multi-semi}.
\end{proof}

\subsection{Symmetric  semi-algebraic sets}
\label{subsec:symmetric}
In this paper we are mostly concerned with semi-algebraic sets which are symmetric. 
In order to define symmetric semi-algebraic sets we first need some more notation.

\begin{notation}
  Let $\mathbf{k}= (k_{1} , \ldots ,k_{\omega}) \in
  \Z_{\geq 0}^{\omega}$, with $k = |\kk| := \sum_{i=1}^{\omega} k_{i}$,  and let $S$ be a
  semi-algebraic subset of $\R^{k}$,
  such that the product of  symmetric groups
  \[
  \mathfrak{S}_{\mathbf{k}}
  :=\mathfrak{S}_{k_{1}} \times \cdots \times \mathfrak{S}_{k_{\omega}}
  \]
  acts on
  $\R^{k}$ by independently permuting each block of coordinates. If $S$ is closed under this action of $\mathfrak{S}_{\mathbf{k}}$, then we say that $S$ is a $\mathfrak{S}_{\mathbf{k}}$-symmetric semi-algebraic set.   We will denote by
  $X/\mathfrak{S}_{\mathbf{k}}$ the \emph{orbit space} of this action. Note that for any symmetric semi-algebraic set $S\subset\R^k$ the corresponding  orbit space  $S/\mathfrak{S}_{\mathbf{k}}$ can be constructed as the image of a polynomial map and thus  is again semi-algebraic (for details see \cite{brocker1998symmetric,Procesi-Schwarz}).
   If
  $\omega =1$, then $k=k_{1}$, and we will denote $\mathfrak{S}_\kk$
  simply by $\mathfrak{S}_{k}$.
\end{notation}

\begin{notation}
\label{not:multisymmetric-polynomial}
Let $\mathbf{k}= (k_{1} , \ldots ,k_{\omega}) \in \Z_{>0}^{\omega}$, with $k= |\kk|$.

We will denote by $\R[\X^{(1)},\ldots,\X^{(\omega)}]^{\mathfrak{S}_{\kk}}_{\leq \dd}$ 
the set of polynomials which are  fixed under the action of $\mathfrak{S}_{\mathbf{k}}
  =\mathfrak{S}_{k_{1}} \times \cdots \times \mathfrak{S}_{k_{\omega}}$
 acting by independently permuting each block of variables $\X^{(i)}$.
 In the case $\omega=1$, $k_1 = k$, $\dd = (d)$, we will denote  
 $\R[\X^{(1)}]^{\mathfrak{S}_{\kk}}_{\leq \dd}$ simply by $\R[X_1,\ldots,X_k]^{\mathfrak{S}_k}_{\leq d}$.
\end{notation}

\subsection{Equivariant cohomology} 
We recall here a few basic facts about equivariant cohomology. 

The important point of the following discussion is that in the setting of the current paper,
for $G$-symmetric semi-algebraic subsets $S \subset \R^k$ (where $G$ is a product of symmetric groups), the 
$G$-equivariant cohomology groups of $S$ with coefficients in a field $\F$ of characteristic $0$,  are isomorphic to 
the singular cohomology of the quotient $S/G$ with coefficients in $\F$ (cf. \eqref{eqn:iso}). 
Thus, bounding the Betti numbers of $S/G$ is equivalent to bounding the $G$-equivariant Betti numbers of $S$.

More precisely, recall that given a  topological space $X$ together with a topological action of  an arbitrary compact Lie group $G$,  one defines the  \emph{equivariant cohomology groups} starting from  a \emph{universal principal $G$-space}, denoted $E G$,
which is contractible, and on which the group $G$ acts freely on the right. The orbit space of this action  is called  the \emph{classifying space} $B G$,  i.e., we have
$B G= E G/G$.

\begin{definition}
  \label{def:equivariant-cohomology} (Borel construction) 
  Let $X$ be a space with a left action of  the group $G$. 
  Then, $G$ acts diagonally on the space $E G \times X$ by $g (z,x) = (z \cdot g^{-1} ,g \cdot x)$. For any
  field of coefficients $\F$, the \emph{$G$-equivariant cohomology
  groups of $X$} with coefficients in $\F$, denoted by
  $\HH^{\ast}_{G} (X,\F)$, is defined by
$\HH^{\ast}_{G} (X,\F)  =  \HH^{\ast} (E G \times X/G,\F)$.
\end{definition}

In the situation of interest in the current paper, where $G= \mathfrak{S}_\kk$ acting on
a $\mathfrak{S}_\kk$-symmetric semi-algebraic subset $S \subset \R^k$,
and $\F$ is a field with characteristic equal to $0$,
we have the isomorphisms (see \cite{BC2013}):

\begin{equation}
\label{eqn:iso}
\HH^{\ast} (S/\mathfrak{S}_\kk,\F) \xrightarrow{\sim} \HH_{\mathfrak{S}_\kk}^{\ast} (S,\F
) \xrightarrow{\sim} \HH^\ast(S,\F)^{\mathfrak{S}_\kk}.
\end{equation}
Therefore, the equivariant Betti numbers are precisely the Betti numbers of the orbit space $S/\mathfrak{S}_\kk$,
and
we will state all the results in the paper in terms of the ordinary Betti numbers of the orbit space.

As mentioned before, equivariant Betti numbers of symmetric real varieties and semi-algebraic sets were studied from a quantitative point of view in \cite{BC2013}. We summarize below the main results proved 
there.

\subsection{Previous Results}
Even though the following result was stated in \cite{BC2013} more generally, with multiple blocks of variables, for ease of reading we state a simplified version having only one block.

Let $S \subset \R^{k}$
  be a $\mathcal{P}$-closed-semi-algebraic set,  where 
  $\mathcal{P} \subset \R[X_1,\ldots, X_k]^{\mathfrak{S}_k}_{\leq d}$, 
  with $\deg (P) \leq  d$ for each $P  \in  \mathcal{P}$, $\card (
  \mathcal{P}) =s$. Then, for all sufficiently large $k > 0$, and any field
 field of coefficients $\F$:
 
\begin{theorem}
  \label{thm:main-sa-closed} 
  \begin{enumerate}[1.]
  \item (Vanishing)
  For all $  i \geq 5 d$,
 \begin{eqnarray*}
    \HH^{i} (S/\mathfrak{S}_{\mathbf{k}} ,\F) & \cong  & 0; 
  \end{eqnarray*}
  \item (Quantitative bound)
  \begin{eqnarray*}
    b (S/\mathfrak{S}_{\mathbf{k}} ,\F) & \leq & s^{5d-1} (O(k))^{4d-1}.
  \end{eqnarray*}
  \end{enumerate}
  \end{theorem}

The main tools that are used in the proof of  Theorem \ref{thm:main-sa-closed} are the following:

\begin{enumerate}[1.]
\item
\label{itemlabel:intro:1}
Infinitesimal equivariant deformations of symmetric varieties, such that the deformed varieties are symmetric,  and moreover has good algebraic and Morse-theoretic properties (isolated, non-degenerate critical points with respect to the first elementary symmetric function, namely  $e_1^{(k)}(X_1,\ldots,X_k) = \sum_{i=1}^{k}
 X_i$) \cite[\S 4, Proposition 4]{BC2013};
\item
\label{itemlabel:intro:2}
Certain equivariant Morse-theoretic results to quantify changes in the equivariant Betti numbers
at the critical points of a symmetric Morse function \cite[\S 4, Lemmas 6, 7]{BC2013};
\item
\label{itemlabel:intro:3}
A bound on the number of distinct coordinates of isolated real solutions of any real symmetric
polynomial system in terms of the degrees of the polynomials \cite[\S 4, Proposition 5]{BC2013}, 
which leads to a polynomial bound on the number of orbits of the set of critical points.
\end{enumerate}

It was remarked in \cite{BC2013}, that the vanishing results as well as the upper bounds
are perhaps not optimal. In particular, item \eqref{itemlabel:intro:1} in the above list (equivariant deformation) already requires a doubling of the degrees of the polynomials involved mainly for a technical reason in order to prove non-degeneracy of the critical points. 

In this paper, we improve both the vanishing result as well as the exponent of the bounds
in Theorem \ref{thm:main-sa-closed} using a completely different
approach that does not rely on Morse theory. We utilize instead certain theorems of 
Kostov \cite{Kostov}, 
Arnold \cite{Arnold},  and
Giventhal \cite{Giventhal} 
(see Theorems \ref{thm:Kostov},  \ref{thm:arnold}, and \ref{thm:Weyl-diffeo} below)
on the level sets of power sum polynomials.

Our main quantitative results are the following. We separate the vanishing part from the quantitative part
for clarity.

\subsection{Main Quantitative Results}
\label{subsec:main}
\subsubsection{Vanishing}
\label{subsubsec:vanishing}

\begin{theorem}(Vanishing)
  \label{thm:vanishing} 
  Let $\mathbf{k}= (k_{1} , \ldots ,k_{\omega}), \dd = (d_1,\ldots,d_\omega) \in \Z_{\geq 0}^{\omega}$,
  with  $k= \sum_{i=1}^{\omega} k_{i}$. 
 Let $\mathcal{P} \subset \R[\X^{(1)},\ldots,\X^{(\omega)}]^{\mathfrak{S}_\kk}_{\leq \dd}$  be a finite set,  where for each  $i, 1 \leq i \leq \omega$,  $\X^{(i)}$ is a block of $k_{i}$ variables.
  Let $S \subset \R^k$ be $\mathcal{P}$-closed semi-algebraic set.
  Then, for any  field of coefficients $\F$,
 \begin{eqnarray*}
    \HH^p(S/\mathfrak{S}_\kk ,\F) & = & 0,
  \end{eqnarray*}
for all 
\[
p  \geq   \sum_{i=1}^{\omega} \min (k_{i},d_{i}).
\] 
\end{theorem}

\begin{remark}
Notice that Theorem \ref{thm:vanishing} 
improves the
corresponding result in Theorem \ref{thm:main-sa-closed}. 
Moreover, the new result is tight (see Remark \ref{rem:vanishing:tightness} for an example).
\end{remark}

\subsubsection{Quantitative Bounds}
\begin{theorem}
 \label{thm:bound}
 Let $S \subset \R^k$ be a $\mathcal{P}$-closed semi-algebraic set, where 
 \[
 \mathcal{P} \subset \R[X_1,\ldots,X_k]^{\mathfrak{S}_k}_{\leq d}, \card(\mathcal{P}) = s, 
 d> 1.
 \] 
 Let
 \begin{eqnarray*}
F(d,k)  &= & 
(2^d -1) \prod_{i=1}^{\lfloor d/2 \rfloor -1}(k - \lceil d/2 \rceil-i) \mbox{ if $d \leq k$}, \\
&\leq& (2^k -1) (k-1)! \mbox{ if $d > k$},
\end{eqnarray*}
and $d' = \min(k,d)$.
 Then,
 \begin{eqnarray*}
 b(S/\mathfrak{S}_k,\F) &\leq &(O(sdd'))^{d'} F(d,k) \\
                                      &=&   d^{O(d)} s^d k^{\lfloor d/2 \rfloor -1} \mbox{ if }  1 < d  \ll s,k.
 \end{eqnarray*}
 
 In particular, if  $\card(\mathcal{P}) = 1$, and $S = \ZZ(\mathcal{P},\R^k)$,  and $1 < d \ll  k$, then 
 \begin{equation}
 \label{eqn:thm:bound:1}
 b(S/\mathfrak{S}_k,\F) \leq  d^{O(d)} k^{\lfloor d/2 \rfloor -1}.
 \end{equation}
 \end{theorem}
 \begin{remark}
Notice that the bounds in Theorem \ref{thm:bound} are better than the corresponding bound
in Theorem \ref{thm:main-sa-closed} in the case of fixed $d$ and $s,k \rightarrow \infty$. 
Also it should be noted that the exponent in the bound given in Theorem \ref{thm:bound} is the same for $d$ and $d+1$, if $d$ is even. 

Finally, with regards to tightness, note that for fixed $d$ and large $s,k$, 
the bound in Theorem \ref{thm:bound}, takes the form $d^{O(d)} s^d k^{\lfloor d/2 \rfloor -1}$, and 
neither of the two exponents (i.e the exponent of $s$ which is equal to $d$, and the exponent of $k$ which is equal to $\lfloor d/2\rfloor -1$) in the bound can be improved. In the case of $s$ this follows from the example in \cite[Remark 7]{BC2013}, and in the case of $k$ this follows from
Example \ref{eg:key}. 

\end{remark}

 In the case of multiple blocks we have the following bound (notice that the field of coefficients $\F = \Z_2$ in the 
 following theorem). 
 \begin{theorem}
  \label{thm:bound-general} 
  Let $\mathbf{k}= (k_{1} , \ldots ,k_{\omega}), \dd = (d_1,\ldots,d_\omega) \in
  \Z_{\geq 0}^{\omega}, 
  \dd > 1^\omega
  $, 
  with  $k= |\kk|$. Let $\mathcal{P} \subset \R[\X^{(1)},\ldots,\X^{(\omega)}]^{\mathfrak{S}_\kk}_{\leq \dd}$  be a finite set of polynomials with
  $\card(\mathcal{P}) = s$.
  Let $S \subset \R^k$ be $\mathcal{P}$-closed semi-algebraic set.
  
 Then,
 \[
 b(S/\mathfrak{S}_k,\Z_2) \leq  
  \left(\prod_{i=1}^{\omega} (O(\omega^3 s d_i d_i'))^{d_i'}F(d_i,k_i)\right),
 \]
 where 
 \begin{eqnarray*}
 d_i' &=& \min(k_i,d_i), 1 \leq i \leq \omega,
 \end{eqnarray*}
 and $F(d_i,k_i)$ as in Theorem \ref{thm:bound}.
 
 \end{theorem}
 
It is worth noticing that requiring a description by symmetric polynomials is not necessary in the case of symmetric real varieties. 
Since every real symmetric variety defined by (possibly non-symmetric) polynomials of degree at most $d$,
can be defined by one symmetric polynomial  of degree at most $2d$
(by taking the sum of squares of all the polynomials appearing in the orbits of the given polynomials under the action of the symmetric group), 
the above results in particular yield the following.  

\begin{corollary}
\label{cor:vanishing}
Let $\mathcal{P}\subset \R[X_1,\ldots,X_k]_{\leq d}$ with $2d\leq k$  such that $\ZZ(\mathcal{P},\R^k)$ is $\mathfrak{S}_k$ invariant, then
 \[
 b(S/\mathfrak{S}_k,\F) \leq  d^{O(d)} 
 k^{d-1},
 \]
and  
 \[ \HH^p(S/\mathfrak{S}_{\mathbf{k}} ,\F)  =  0,\text{ for all }  p  \geq 2d. \]
\end{corollary}

\subsection{Symmetric definable sets in an o-minimal structure}
While the main goal of this paper is to study the equivariant Betti numbers of symmetric semi-algebraic,
the methods developed in this paper for bounding the equivariant Betti numbers  of symmetric semi-algebraic sets actually extend to more general situations. We illustrate this by considering certain classes of symmetric definable sets in an arbitrary o-minimal expansion of the real closed field $\R$ (we refer the reader to \cite{Dries} and \cite{Michel2} for basic facts about o-minimal geometry). In the non-equivariant case, quantitative upper bounds on the Betti numbers of definable sets belonging to the Boolean algebra generated by a finite family of the fibers of some fixed definable map was studied in \cite{Basu9} and tight upper bounds were obtained. Here we consider \emph{symmetric} definable sets which are defined as the pull-back of a (not necessarily symmetric ) definable set under a polynomial map which is symmetric (and of some fixed degree). 
Our methods yield the following theorems.

\begin{theorem}
\label{thm:definable1}
Let $V \subset \R^m$ be a closed definable set in an o-minimal structure over $\R$. Then, for all $d >0$, there exists a constant $C_{V,d} > 0$ such that for all $k \geq d$,  and any  polynomial map $F = (F_1,\ldots,F_m): \R^k \rightarrow \R^m$, where $F_i \in \R[X_1,\ldots,X_k]^{\mathfrak{S}_k}_{\leq d}, 1 \leq i \leq m$ we have:
\begin{enumerate}[1.]
\item
\label{itemlabel:thm:definable1:1}
the definable set 
$F^{-1}(V) \subset \R^k$ is symmetric;
\item
\label{itemlabel:thm:definable1:2}
$\HH^p(F^{-1}(V)/\mathfrak{S}_k,\F) = 0$ for $p \geq d$; and,
\item
\label{itemlabel:thm:definable1:3}
\[
b(F^{-1}(V)/\mathfrak{S}_k,\F) \leq C_{V,d} \cdot  k^{\lfloor d/2 \rfloor-1}.
\]
\end{enumerate}
\end{theorem}

Following \cite{Basu9} we now define the definable analog of $\mathcal{P}$-closed semi-algebraic sets (cf. Definition \ref{def:sign-condition}).

\begin{definition}[$\mathcal{A}$-closed sets]
\label{def:A-closed}
For any finite family $\mathcal{A} =\{A_1,\ldots,A_s\}$ of definable subsets of $\R^k$, we call a definable subset $S \subset \R^k$ to be an 
\emph{$\mathcal{A}$-closed set},
 if $S$ is a finite union of sets of the form 
\[
\bigcap_{i \in I} A_i 
\]
where $I \subset [1,s]$.
\end{definition}

 The following generalization of Theorem \ref{thm:definable1} holds.

\begin{theorem}
\label{thm:definable2}
Suppose that  $V \subset \R^m \times \R^\ell$ is a closed definable set in an o-minimal structure over $\R$, and $\pi_1: \R^m \times \R^\ell \rightarrow \R^m, \pi_2: \R^m \times \R^\ell \rightarrow \R^\ell$ the two projection maps, and  for $ y \in \R^\ell$ denote by 
$V_y$ the definable set $\pi_1(\pi_2^{-1}(y) \cap V)$. Then for each $d > 0$, there exists a constant $C_{V,d} > 0$, such that for every finite subset $A \subset \R^\ell$, and every $\mathcal{A}$-closed set $S \subset \R^m$, where $\mathcal{A} = \cup_{y \in A} \{V_y\}$, the following holds.

For any $k \geq d$,  and any polynomial map $F = (F_1,\ldots,F_m): \R^k \rightarrow \R^m$, where $F_i \in \R[X_1,\ldots,X_k]^{\mathfrak{S}_k}_{\leq d}$,  $1 \leq i \leq m$ we have:
\begin{enumerate}[1.]
\item
the definable set 
$F^{-1}(S) \subset \R^k$ is symmetric;
\item
$\HH^p(F^{-1}(S)/\mathfrak{S}_k,\F) = 0$ for $p \geq d$; and, 
\item
\[
b(F^{-1}(S)/\mathfrak{S}_k,\F) \leq C_{V,d}\cdot s^d \cdot k^{\lfloor d/2 \rfloor-1},
\]
where $s = \card(A)$.
\end{enumerate}
\end{theorem}

\begin{remark}
\label{rem:definable}
In Theorem \ref{thm:definable2}, if one wants to bound the ordinary Betti numbers of $F^{-1}(S)$, then
an upper bound of the form $b(F^{-1}(S),\F) \leq C_{V,d,k}\cdot s^k$ follows immediately from Theorem 2.3 in \cite{Basu9},
however the constant $C_{V,d,k}$ depends on $k$ and hence the dependence of the bound on $k$ is not explicit. In contrast, in  Theorems \ref{thm:definable1} and \ref{thm:definable2}, the constant $C_{V,d}$ is independent of $k$, and the dependence of the stated bounds on $k$ is explicit.
\end{remark}

\begin{example}
\label{eg:definable}
We now give an illustration of application of Theorem \ref{thm:definable2} for bounding the equivariant Betti numbers of a certain  concrete sequence of 
definable sets in an o-minimal structure larger than the  o-minimal structure of semi-algebraic sets. 
Consider the o-minimal structure $\mathbb{R}_{\mathrm{exp}}$ (the real numbers equipped with the exponential function). 
Theorem \ref{thm:definable2} then implies that for every fixed $m,d > 0$, there exists a constant $C_{m,d} > 0$ such that for any 
$F_1,\ldots,F_m \in \mathbb{R}[X_1,\ldots,X_k]^{\mathfrak{S}_k}_{\leq d}$, and 
$ 
(a_{1,1},\ldots,a_{1,m}),\ldots, 
(a_{s,1},\ldots,a_{s,m}) \in \mathbb{R}^m$, denoting by
$S \subset \mathbb{R}^k$, the \emph{union} of the $s$ definable subsets of $\mathbb{R}^k$ defined by the $s$ equations
\begin{eqnarray*}
a_{1,1} e^{F_1} + \cdots + a_{1,m} e^{F_m} &=& 0, \\
\vdots&\vdots &\vdots\\
a_{s,1} e^{F_1} + \cdots + a_{s,m} e^{F_m} &=& 0,
\end{eqnarray*}
the inequality
\[
b(S/\mathfrak{S}_k,\F) \leq C_{m,d} \cdot s^d \cdot k^{\lfloor d/2 \rfloor-1}
\]
holds.
\end{example}

\subsection{Algorithm}
\label{subsec:algo-results}
An important consequence of our new method is that we also obtain an algorithm with \emph{polynomially bounded  complexity} (for every fixed degree) for computing the rational equivariant Betti numbers of closed, symmetric semi-algebraic subsets of $\R^k$. This answers
a question posed in \cite{BC2013}. 

More precisely,
it was asked  in \cite{BC2013} whether there exists for every fixed $d$, an 
algorithm for computing the equivariant Betti numbers of symmetric $\mathcal{P}$-closed semi-algebraic subsets
of $\R^k$, where $\mathcal{P} \subset \R[X_1,\ldots,X_k]^{\mathfrak{S}_k}_{\leq d}$, and
whose complexity is bounded polynomially in $\card(\mathcal{P})$ and $k$ (for constant $d$).
Using the method of equivariant deformation and equivariant Morse theory, an 
algorithm with polynomially bounded complexity for computing (both the equivariant as well as the ordinary)
Euler-Poincar\'e characteristics of symmetric algebraic sets appears in \cite{BC2013-ep}. However, this method does not extend to an 
algorithm for computing all the equivariant Betti numbers, and indeed it is well known that the algorithmic 
problem of computing the Euler-Poincar\'e
characteristic is simpler than that of computing all the individual Betti numbers.

In the classical Turing machine model the problem of computing Betti numbers (indeed just the number of 
connected components) of a real variety defined by a polynomial of degree $4$ is $\mathbf{PSPACE}$-hard
\cite{Reif79}.
On the other hand it follows from the existence of doubly exponential algorithms for semi-algebraic triangulation
(see \cite{BPRbook2} for definition)
of real varieties, 
that there also exist algorithms with  doubly exponential complexity for computing the Betti numbers of real varieties and semi-algebraic sets
\cite{SS}.
The following theorem that we prove in this paper shows that the equivariant case is markedly different from the point of view of algorithmic complexity.

\begin{theorem}
\label{thm:algorithm}
For every fixed $d \geq 0$, there exists an algorithm that takes as input a 
$\mathcal{P}$-closed formula $\Phi$, where $\mathcal{P} \subset \R[X_1,\ldots,X_k]^{\mathfrak{S}_k}_{\leq d}$, and outputs 
$b^i(S/\mathfrak{S}_k,\F), 0 \leq i < d$, where $S = \RR(\Phi,\R^k)$. The complexity of this algorithm  is bounded by 
$(\card(\mathcal{P}) k d)^{2^{O(d)}}$.
 \end{theorem}

\begin{remark}
\label{rem:algorithm}
Notice that for fixed $d$ the complexity of the algorithm in Theorem \ref{thm:algorithm} is polynomial in $\card(\mathcal{P})$ and $k$.
\end{remark}

\section{Proofs of the main theorems} 

\subsection{Outline of the proofs}
\label{subsec:outline}
As mentioned in the Introduction the main ideas behind the proofs of Theorems \ref{thm:vanishing}, \ref{thm:bound}, and \ref{thm:bound-general} are quite different from the Morse theoretic arguments used in \cite{BC2013}. For convenience of the reader we outline the main ideas that are used first.

In order to prove  to Theorem \ref{thm:vanishing}, we prove directly that if $S \subset \R^k$ is a closed and bounded symmetric semi-algebraic set, defined by symmetric polynomials of degree at most 
$d \leq k$, then $S/\mathfrak{S}_k$ is homologically equivalent to a certain semi-algebraic subset of  $\R^d$ (Part \eqref{itemlabel:prop:vanishing:2} of Proposition \ref{prop:vanishing} below). This immediately implies the vanishing of the higher cohomology groups of  $S/\mathfrak{S}_k$. In order to prove the homological  equivalence we use
certain results on the properties of Vandermonde mappings due to Kostov and Giventhal
(see Theorems \ref{thm:Kostov} and \ref{thm:Weyl-diffeo} below). This argument avoids the technicalities of having to produce a good equivariant deformation required in the Morse-theoretic arguments for proving a similar vanishing result in \cite{BC2013}, which led to a worse
bound on the vanishing threshold in terms of the degrees ($2d$ in the algebraic case, and $5d$ in the semi-algebraic case).

In order to prove  the upper bounds on the equivariant Betti numbers of symmetric semi-algebraic sets  (Theorems \ref{thm:bound} and \ref{thm:bound-general})  we prove first that
if $S \subset \R^k$ is a closed and bounded symmetric semi-algebraic set, defined by symmetric polynomials of degree at most 
$d \leq k$, then 
$S/\mathfrak{S}_k$, is homologically equivalent to the intersection, $S_{k,d}$, of $S$ with a 
certain polyhedral complex of dimension $d$ in $\R^k$ (Proposition \ref{prop:vanishing}) -- namely, the subcomplex formed by certain $d$-dimensional faces of the Weyl chamber defined by
$X_1 \leq X_2 \leq \cdots \leq X_k$ 
(cf. Propositions \ref{prop:vanishing} and \ref{prop:arnold}).
Thus, in order to bound the Betti numbers of $S/\mathfrak{S}_k$, it suffices to bound the Betti numbers of
$S_{k,d}$ (see Part \eqref{itemlabel:prop:arnold:c} of Proposition \ref{prop:arnold}).

The number of $d$-dimensional faces of the Weyl chamber that we need to consider 
is $$\binom{k -\lceil d/2\rceil-1}{\lfloor d/2 \rfloor -1} = (O_d(k))^{\lfloor d/2 \rfloor-1}.$$ Since the intersection of each one of these faces with $S$
is contained in a linear subspace of dimension $d$, the Betti numbers of such intersections can be bounded by a polynomial in $s, k$ of degree $d$ (cf. Corollary \ref{cor:multi-semi}). Moreover, the intersections amongst these sets are themselves intersections of $S$ with faces of the Weyl chamber of smaller dimensions. We then use inequalities coming from the  Mayer-Vietoris spectral sequence (cf. Proposition \ref{prop:MV}) to obtain a bound on $S/\mathfrak{S}_k$. However, a straightforward argument using Mayer-Vietoris inequalities will produce a much worse bound than claimed in  Theorems \ref{thm:bound} and \ref{thm:bound-general}. This is because the number of possibly non-empty intersections that needs to be accounted for would be too large.
In order to control this combinatorial part we use an argument involving infinitesimal thickening and shrinking of the faces of the Weyl chambers. Such perturbations involve extending the field $\R$, to a real closed field of Puiseux series in the infinitesimals that are introduced with coefficients in $\R$. We recall some basic facts about fields of Puiseux series in \S \ref{subsubsec:puiseux}. After replacing the faces of the Weyl chambers by certain new sets defined in terms of  infinitesimal thickening and shrinking, we show that
only flags (not necessarily complete flags) of faces contribute to the Mayer-Vietoris inequalities 
(Corollary \ref{cor:chains}). The number of such flags is bounded by $(O_d(k))^{\lfloor d/2 \rfloor-1}$ (cf. Proposition \ref{prop:chainsCompcat}).
This together with bounds on the Betti numbers of semi-algebraic sets in terms of the multi-degrees of the defining polynomials (cf. Corollary  \ref{cor:multi-semi})  lead to the claimed bounds.
In the o-minimal category (proofs of Theorems \ref{thm:definable1} and \ref{thm:definable2}), we follow the same strategy, except the explicit bounds on the Betti numbers as in Corollary  \ref{cor:multi-semi} are replaced by bounds involving a constant that depends on the given definable family (the dependence of the other parameters remain the same as in the semi-algebraic case). Since these proofs are quite
similar to the ones in the semi-algebraic case, we only give a sketch of the arguments indicating the modifications that need to be made from the semi-algebraic case.

\subsection{Preliminaries}
\label{subsec:prelim}
In this section we recall some basic facts about real closed fields and real
closed extensions.

\subsubsection{Real closed extensions and Puiseux series}
\label{subsubsec:puiseux}
We will need some
properties of Puiseux series with coefficients in a real closed field. We
refer the reader to {\cite{BPRbook2}} for further details.

\begin{notation}
  For $\R$ a real closed field we denote by $\R \left\langle \eps
  \right\rangle$ the real closed field of algebraic Puiseux series in $\eps$
  with coefficients in $\R$. We use the notation $\R \left\langle \eps_{1} ,
  \ldots , \eps_{m} \right\rangle$ to denote the real closed field $\R
  \left\langle \eps_{1} \right\rangle \left\langle \eps_{2} \right\rangle
  \cdots \left\langle \eps_{m} \right\rangle$. Note that in the unique
  ordering of the field $\R \left\langle \eps_{1} , \ldots , \eps_{m}
  \right\rangle$, $0< \eps_{m} \ll \eps_{m-1} \ll \cdots \ll \eps_{1} \ll 1$.
\end{notation}

\begin{notation}
\label{not:lim}
  For elements $x \in \R \left\langle \eps \right\rangle$ which are bounded
  over $\R$ we denote by $\lim_{\eps}  x$ to be the image in $\R$ under the
  usual map that sets $\eps$ to $0$ in the Puiseux series $x$.
\end{notation}

\begin{notation}
\label{not:extension}
  If $\R'$ is a real closed extension of a real closed field $\R$, and $S
  \subset \R^{k}$ is a semi-algebraic set defined by a first-order formula
  with coefficients in $\R$, then we will denote by $\Ext(S, \R') \subset \R'^{k}$ the semi-algebraic subset of $\R'^{k}$ defined by
  the same formula. It is well-known that $\Ext(S, \R')$ does
  not depend on the choice of the formula defining $S$ {\cite{BPRbook2}}.
\end{notation}

\begin{notation}
\label{not:ball}
  For $x \in \R^{k}$ and $r \in \R$, $r>0$, we will denote by $B_{k} (x,r)$
  the open Euclidean ball centered at $x$ of radius $r$. If $\R'$ is a real
  closed extension of the real closed field $\R$ and when the context is
  clear, we will continue to denote by $B_{k} (x,r)$ the extension $\Ext(B_{k} (x,r) , \R')$. This should not cause any confusion.
\end{notation}

\subsection{Mayer-Vietoris inequalities}
We will need the following inequalities. They
are consequences of  Mayer-Vietoris exact sequence.

Let $S_{1} , \ldots ,S_{N} \subset \R^{k}$, $N \ge 1$, be closed
semi-algebraic subsets of $\R^{k}$. For $J \subset [1,n]$, we denote 
\begin{eqnarray*}
S_J &=& \bigcap_{j \in J} S_j, \\
S^J &=& \bigcup_{j \in J} S_j.
\end{eqnarray*}

\begin{proposition}
  \label{prop:prop1}
   \begin{enumerate}[1.]
    \item 
    \label{itemlabel:prop:prop1:1}
    For $i \geq 0$,
    
     \begin{equation}
     \label{eqn:prop1:1}
      b^{i} (S^{[1,s]} ,\F) \leq \sum_{j=1}^{i+1}
      \sum_{\substack{
        J \subset \{ 1, \ldots ,s \}\\
        \card (J) =j
        }}
       b^{i-j+1} (S_{J} ,\F) .
    \end{equation}
    
    \item 
    \label{itemlabel:prop:prop1:2}
 \begin{equation}
      \label{eqn:prop1} 
      b^{i} (S_{[1,s]} ,\F) \leq \sum_{j=1}^{k-i}
      \sum_{\substack{
        J \subset \{ 1, \ldots ,s \}\\
        \card (J) =j
        }}
        b^{i+j-1} (S^{J} ,\F) + \binom{s}{k-i} b^{k}
      (S^{\emptyset} ,\F) .
    \end{equation}
  \end{enumerate}
  \end{proposition}

\begin{proof} See 
\cite[Proposition 7.33]{BPRbook2}.
\end{proof}

We will also need the following inequality that is a simple consequence of the Mayer-Vietoris exact sequence. Let $S_1,S_2 \subset \R^k$ be closed, semi-algebraic sets. Then for every $p \geq 0$,
\begin{eqnarray}
\label{eqn:MV-two-sets}
b^p(S_1,\F)+b^p(S_2,\F) \leq b^p(S_1\cup S_2,\F) + b^p(S_1 \cap S_2,\F).
\end{eqnarray} 

\subsection{Bounds on the Betti numbers of $\mathcal{P}$-closed semi-algebraic sets}
\label{subsec:ordinary-Betti}
In order to get the desired bounds using the technique outlined in \S \ref{subsec:outline}
we need to refine slightly  some arguments in \cite[Chapter 7]{BPRbook2} on bounding the Betti numbers of closed semi-algebraic sets. We explain these refinements in the current section.
The main results that will be needed later are Propositions \ref{prop:inductive} and \ref{prop:P-closed-main}.

We begin with:

\begin{proposition}
\label{prop:inductive}
Let $V\subset \R^k$ be a closed semi-algebraic set and $\mathcal{L} \subset \R[X_1,\ldots,X_k]$ a finite set of polynomials, and let 
$S = \{\x \in V \mid \bigwedge_{L \in \mathcal{L}'} L(\x) \geq 0\}$. Then, for every $p \geq 0$, and any field $\F$,
\begin{eqnarray*}
b^p(S,\F) &\leq& \sum_{\mathcal{L}' \subset \mathcal{L}} b^p(V \cap \ZZ(\mathcal{L}',\R^k),\F) .
\end{eqnarray*}
\end{proposition}

\begin{proof}
Let $\mathcal{L} = \{L_1,\ldots,L_m\}$, and
let for $I \subset [1,m]$, 
\begin{eqnarray*}
W_I &=& \RR(\bigwedge_{i \in I} L_i \geq 0, \R^k),\\
Z_I  &=&  \RR(\bigwedge_{i \in I} L_i = 0, \R^k).\\
\end{eqnarray*}

Then, $S= V \cap W_{[1,m]}$.

We prove the statement by induction on $m$. Clearly, the statement is true for $m=0$.
Suppose the statement holds for $m-1$.

Using the induction hypothesis, we have for each $p \geq 0$,
\begin{eqnarray}
\label{eqn:prop:inductive:1}
b^p(V\cap W_{[1,m-1]},\F) &\leq & \sum_{I \subset [1,m-1]} b^p(V \cap Z_I,\F), \\
\label{eqn:prop:inductive:2}
b^p(V \cap Z_{m} \cap W_{[1,m-1]},\F) &\leq & \sum_{I \subset [1,m-1]} b^p(V \cap Z_{I\cup \{m\}}, \F).
\end{eqnarray}

Defining $S' = \{ \x \in  V\cap W_{[1,m-1]} \mid L_m(\x) \leq 0\}$, we have
\begin{eqnarray*}
V\cap W_{[1,m-1]} &=&  S \cup S', \\
V \cap Z_{m} \cap W_{[1,m-1]} &=& S \cap S'.
\end{eqnarray*}

Now, using inequality \eqref{eqn:MV-two-sets}
we have that, for every  $p \geq 0$,
\[
b^p(S,\F) + b^p(S',\F)  \leq b^p(V\cap W_{[1,m-1]},\F) + b^p(V \cap Z_{m} \cap W_{[1,m-1]},\F),
\]
from whence we get,

\begin{eqnarray}
\label{eqn:prop:inductive:3}
b^p(S,\F)  &\leq&  b^p(V\cap W_{[1,m-1]},\F) + b^p(V \cap Z_{m} \cap W_{[1,m-1]},\F).
\end{eqnarray}
The proposition now follows from 
\eqref{eqn:prop:inductive:1}, \eqref{eqn:prop:inductive:2}, and \eqref{eqn:prop:inductive:3}.
\end{proof}

We fix for the remained of the section a closed and semi-algebraically contractible semi-algebraic set $W \subset \R^k$, and
also finite sets
$\mathcal{P} = \{P_1,\ldots,P_s\}, 
\mathcal{F} = \{F_1,\ldots,F_m\} \subset\R^k$.

Let 
\[
\widetilde{W} = \{ \x \in W \mid \bigwedge_{i=1}^{m} F_i(\x) \geq 0\},
\]
and we will also  suppose that $\widetilde{W}$ is semi-algebraically contractible.

Let $\delta_{1} , \cdots , \delta_{s}$ be  infinitesimals, and let
$\R' = \R \langle \delta_{1} , \ldots , \delta_{s} \rangle$.

\begin{notation}
  We define $\mathcal{P}_{>i} = \{P_{i+1} , \ldots ,P_{s} \}$ and
  \begin{eqnarray*}
    \Sigma_{i} & = & \{P_{i} =0,P_{i} = \delta_{i} ,P_{i} = - \delta_{i}
    ,P_{i} \geq 2 \delta_{i} ,P_{i} \leq -2 \delta_{i} \} ,\\
    \Sigma_{\le i} & = & \{\Psi \mid \Psi = \bigwedge_{j=1, \ldots ,i}
    \Psi_{i} , \Psi_{i} \in \Sigma_{i} \} .
  \end{eqnarray*}
  
  If $\Phi$ is a $\mathcal{P}$-closed formula, 
  and $Z \subset \R^k$ a closed semi-algebraic set
  we denote
  \begin{eqnarray*}
    \RR_{i} (\Phi,Z) & = & \RR(\Phi , \R \la \delta_{1} , \ldots ,
    \delta_{i} \ra^{k}) \cap \Ext(Z, \R \la \delta_{1} , \ldots ,
    \delta_{i} \ra^{k}) ,
  \end{eqnarray*}
  and
  \begin{eqnarray*}
    \RR_{i} (\Phi \wedge \Psi,Z ) & = & \RR(\Psi , \R \la \delta_{1} ,
    \ldots , \delta_{i} \ra^{k}) \cap \RR_{i} (\Phi) \cap \Ext(Z, \R \la \delta_{1} , \ldots ,
    \delta_{i} \ra^{k}).
  \end{eqnarray*}
  Finally, we denote for each $\mathcal{P}$-closed formula $\Phi$
  \begin{eqnarray}
  \label{eqn:def:b-Phi-Z}
    b (\Phi,Z,\F) & = & b(\RR
   (\Phi, Z) ,\F
   ) .
  \end{eqnarray}
\end{notation}

The proof of the following proposition is very similar to Proposition 7.39 in
{\cite{BPRbook2}} where it is proved in the non-symmetric case.

\begin{proposition}
  \label{prop:closed-with-parameters}
  For every $\mathcal{P}$-closed formula
  $\Phi$,
  such that $\RR(\Phi,\R^k)$ is bounded,
  \begin{eqnarray*} b (\Phi,Z,\F) &\leq&
     \sum_{\substack{
       \Psi \in \Sigma_{\le s}\\
       \RR_{s} (\Psi, \R'^{k}) \subset \RR_{s} (\Phi , \R'^{k})}}
      b (\Psi,Z,\F).
      \end{eqnarray*}
\end{proposition}

\begin{proof} See Proposition 7.39 in {\cite{BPRbook2}}.
\end{proof}

For $1 \leq i \leq s$, let
\[
Q_{i} =P_{i}^{2} (P_{i}^{2} - \delta_{i}^{2})^{2} (P_{i}^{2} -4
\delta_{i}^{2}),
\]

and for $I  \subset [1,s]$ let,
\begin{eqnarray}
\label{eqn:VI}
  V^I  & = & \RR(\bigvee_{i \in I} Q_{i} =0, \R'^k)  \cap \Ext(\widetilde{W}, \R'^{k}),\\
  \label{eqn:TI}
  T^I & = & \RR(\bigvee_{i \in I} Q_{i} \geq 0, \R'^k) \cap \Ext(\widetilde{W}, \R'^{k}) .
\end{eqnarray}

\begin{proposition}
  \label{prop:betti closed}
  
  For $p \geq 0$,
  \begin{eqnarray}
  \nonumber
   \sum_{\Psi \in \Sigma_{\le s}} b^p (\Psi,\widetilde{W}
     ,\F) &\leq & 
     \sum_{\ell=1}^{k-p} \sum_{I \subset [1,s],\card(I) = \ell}  b^{p+\ell-1}(T^I,\F)\\
     \label{eqn:prop:betti closed}
  &=&  \sum_{I \subset [1,s]}  b^{p+\card(I)-1}(T^I,\F).
  \end{eqnarray}
  \end{proposition}
  
  \begin{proof}
  From \eqref{eqn:def:b-Phi-Z} we have that 
  $b^p (\Psi,\widetilde{W},\F) = b^p (\RR(\Psi,\widetilde{W}),\F)$, and it follows from the definition of $\Psi$, that
  $\RR(\Psi,\widetilde{W})$ is a disjoint union of closed semi-algebraic subsets of the  closed semi-algebraic 
  set 
  \[
  \RR(\bigwedge_{i \in [1,s]} Q_{i} \geq 0, \R'^k) \cap \Ext(\widetilde{W}, \R'^{k}).
  \]
  The proposition now follows from Part \eqref{itemlabel:prop:prop1:2} of Proposition \ref{prop:prop1}, and 
  \eqref{eqn:TI}.
 \end{proof}

\begin{lemma}
  \label{lem:union1}
 \begin{eqnarray*}
   b^p (T^I,\F) &\leq& b^p(V^I,\F), \mbox{ if $p > 0$}, \\
   b^0(T^I,\F) &\leq & b^0(V^I,\F) +1.
   \end{eqnarray*}
 \end{lemma}

\begin{proof} Let
\[ 
Z^I = \RR (\bigwedge_{1 \leq i \leq j} Q_{i} \leq 0 \vee \bigvee_{1 \leq
   i \leq j} Q_{i} =0, \R \la \delta_{1} , \ldots , \delta_{j} \rangle)^k
 )\cap \Ext(W,\R \la \delta_{1} , \ldots , \delta_{j} \rangle) . \]

Clearly 
\[ 
T^I \cup Z^I= \Ext(\widetilde{W},\R \la \delta_{1} , \ldots , \delta_{j} \rangle),
T^I \cap Z_I =V^I.
\] 
The lemma now follows from inequality \eqref{eqn:MV-two-sets}, using the fact that $\widetilde{W}$ is 
semi-algebraically contractible.
\end{proof}

\begin{lemma}
  \label{lem:union2} For each $p \geq 0$,
 \begin{eqnarray*}
  b^p (V^I,\F) &\leq &  \sum_{\ell=1}^{p+1} \sum_{\substack{J \subset I, \\\card(J) =\ell}} 
  \sum_{\tau \in \{0,\pm 1,\pm 2\}^J} b^{p-\ell+1} (\ZZ(\mathcal{P}_{\tau},\R'^k) \cap \Ext(\widetilde{W},\R'),\F) \\
  &=&
   \sum_{J \subset I} 
  \sum_{\tau \in \{0,\pm 1,\pm 2\}^J} b^{p-\card(J)+1} (\ZZ(\mathcal{P}_{\tau},\R'^k) \cap \Ext(\widetilde{W},\R'),\F),
  \end{eqnarray*}
  where 
  \begin{eqnarray}
  \label{eqn:P-tau}
  \mathcal{P}_\tau &=& \bigcup_{j \in J} \{P_j + \tau(j) \delta_j\}.
  \end{eqnarray}
\end{lemma}

\begin{proof}
Let for $i \in [1,s]$, $V_i = \ZZ(Q_i,\R'^k) \cap \Ext(\widetilde{W},\R'^k)$. Then, for each $i \in [1,s]$,
$V_i$
is the disjoint union of the following  five sets,
\begin{align*}
\ZZ(P_{i}, \R'^k )&~\medcap~\Ext(\widetilde{W},\R'^k), \\
\ZZ (P_{i} \pm\delta_{i}, \R'^{k})&~\medcap ~\Ext(\widetilde{W},\R'^k),\\
\ZZ (P_{i} \pm 2 \delta_{i}, \R'^{k})&~\medcap~\Ext(\widetilde{W},\R'^k).
 \end{align*}
The lemma now follows from Part     \eqref{itemlabel:prop:prop1:1} of Proposition \ref{prop:prop1}.
\end{proof}

\begin{proposition}
\label{prop:P-closed}
For every $\mathcal{P}$-closed formula
  $\Phi$,
  such that $\RR(\Phi,\R^k)$ is bounded,
  \begin{equation} 
  \label{eqn:prop:P-closed}
  b (\Phi,\widetilde{W},\F) \leq     1+s +\sum_{ p \geq 0} \sum_{\substack{ I \subset [1,s],\\ 1 \leq \card(I) \leq k-p, \\
     J \subset I,\\ 1\leq \card(J) \leq p+1}} \sum_{\tau \in \{0,\pm 1,\pm 2\}^J}  F(p,\card(I),J,\tau),
     \end{equation}
     where
     \begin{eqnarray}
     F(p,q,J,\tau),
     &=&
     b^{p+q -\card(J)} (\ZZ(\mathcal{P}_{\tau},\R'^k) \cap \Ext(\widetilde{W},\R'),\F).
     \end{eqnarray}
\end{proposition}

\begin{proof}
The proposition follows from Propositions \ref{prop:closed-with-parameters} and \ref{prop:betti closed}, and
Lemmas \ref{lem:union1} and \ref{lem:union2}, after noting that on the right side of 
\eqref{eqn:prop:betti closed} in Proposition \ref{prop:betti closed}, $p + \card(I) -1 = 0$ implies that
$\card(I) = 0$ or $1$ since $p \geq 0$. This accounts for the additive factor of $1+s$ on the right side of 
\eqref{eqn:prop:P-closed}.
\end{proof}

Finally, using the same notation as Proposition \ref{prop:P-closed}:

\begin{proposition}
\label{prop:P-closed-main}
For every $\mathcal{P}$-closed formula
  $\Phi$,  such that $\RR(\Phi,\R^k)$ is bounded,
  \begin{eqnarray*} b (\Phi,\widetilde{W},\F) &\leq&
    1+ s+\sum_{ p \geq 0} \sum_{\substack{ I \subset [1,s],\\ 1 \leq \card(I) \leq k-p, \\
     J \subset I,\\ 1\leq \card(J) \leq p+1}} \sum_{\sigma \in \{0,\pm 1,\pm 2\}^J} 
     \sum_{K \subset [1,m]}  G(p,\card(I),J,K,\sigma),
     \end{eqnarray*}
     where
     \begin{eqnarray*}
     G(p,q,J,K,\sigma),
     &=&
     b^{p+q -\card(J)} (\ZZ(\mathcal{P}_{\sigma},\R'^k)
     \cap \widetilde{V}_K,\F),
     \end{eqnarray*}
where for $K \subset [1,m]$, 
\[\widetilde{V}_K = W  \bigcap_{i \in K} \ZZ(F_i,\R^k).
\]
 \end{proposition}

\begin{proof}
Use Propositions \ref{prop:P-closed} and \ref{prop:inductive}.
\end{proof}

\subsection{Proof of Theorem \ref{thm:vanishing}}
Before proving Theorem \ref{thm:vanishing} we need a preliminary result.

We first need some notation.

\begin{notation}
\label{not:Weyl}
Let $\Weyl^{(k)} \subset \R^k$ denote the cone defined by $X_1 \leq X_2 \leq \cdots \leq X_k$. 

More generally, for $\kk = (k_1,\ldots,k_\omega) \in \Z_{>0}^\omega$, we will denote 
\[
\boldsymbol{\Weyl}^{(\kk)} = \Weyl^{(k_1)} \times \cdots \times \Weyl^{(k_\omega)}.
\]

For every $m \geq 0$, and $\w = (w_1,\ldots,w_k) \in \R_{> 0}^k$, let $p_{\w,m}^{(k)}:\Weyl^{(k)} \rightarrow \R$
be the polynomial  map defined by:
$$\displaylines{
\forall  \x = (x_1,\ldots,x_k) \in \Weyl^{(k)}, \cr
 p_{\w,m}^{(k)}(\x) = \sum_{j=1}^{k} w_j x_j^m.
}
$$

For every $ d\geq 0 $, and $\w \in \R_{> 0}^k$ we denote by 
$\Psi_{\w,d}^{(k)}:\Weyl^{(k)} \rightarrow \R^{d'}$,
the continuous map defined by 
$$
\displaylines{
\forall  \x = (x_1,\ldots,x_k) \in \Weyl^{(k)}, \cr
 \Psi_{\w,d}^{(k)}(\x) = (p_{\w,1}^{(k)}(\x),\ldots,p_{\w,d'}^{(k)}(\x)),
}
$$
where $d' = \min(k,d)$.

If $\w = 1^k :=  (1,\ldots,1)$, then we will denote by $p_m^{(k)}$ the polynomial 
$p_{\w,m}^{(k)}$  (the $m$-th Newton sum polynomial), and by $\Psi^{(k)}_d$ the map $\Psi^{(k)}_{\w,d}$.
\end{notation}

We will need the following theorem due to Kostov.
\begin{theorem}\cite[Theorem 1]{Kostov}
\label{thm:Kostov}
For every $\w \in \R_{\geq 0}^k$,  $d,k \geq 0$, and $\y \in \R^d$, 
the fibre $$V_{\w,d,\y} := (\Psi^{(k)}_{\w,d})^{-1}(\y)$$ is either empty or contractible.
\end{theorem}

We will also need:

\begin{theorem}\cite[first Corollary]{Giventhal}
\label{thm:Weyl-diffeo}
The map 
$\Psi_{k}^{(k)}:\Weyl^k \rightarrow \R^k$
  is a 
homeomorphism
on to its image.
\end{theorem}

As an immediate corollary of Theorem \ref{thm:Kostov} we have:

\begin{corollary}
\label{cor:Kostov}
Let 
$$
\displaylines{ \kk = (k_1,\ldots, k_\omega) \in \Z_{\geq0 }, \cr
\dd = (d_1,\ldots,d_\omega) \in \Z_{\geq 0}, \cr
d_i' = \min(k_i,d_i), 1 \leq i \leq \omega.
}
$$
Let 
\[
\boldPsi_{\dd}^{(\kk)}: 
\boldsymbol{\Weyl}^{(\kk)}
\rightarrow \R^{d_1'} \times \cdots \times \R^{d_\omega'}
\]
 denote the map defined by 
$$\displaylines{
\forall \x=(\x^{(1)},\ldots,\x^{(\omega)}) \in \boldsymbol{\Weyl}^{(\kk)}, \cr
\boldPsi_{\dd}^{(\kk)}(\x^{(1)},\ldots,\x^{(\omega)})  = 
(\Psi_{d_1'}^{(k_1)}(\x^{(1)}),\ldots,  \Psi_{d_\omega'}^{(k_\omega)} (\x^{(\omega)})).
}
$$
Then,
for each $ \y \in \R^{d_1'} \times \cdots \times \R^{d_\omega'}$,
$(\boldPsi_{\dd}^{(\kk)})^{-1}(\y)
$ 
is either empty or contractible.
\end{corollary}

We will need the following proposition. With the same notation as in Theorem \ref{thm:vanishing}:

\begin{proposition}
\label{prop:vanishing}
Let $\mathcal{P} \subset \R[\X^{(1)},\ldots,\X^{(\omega)}]^{\mathfrak{S}_\kk}_{\leq \dd}$  and let $S\subset \R^k$ be a 
bounded
$\mathcal{P}$-closed
semi-algebraic set. 

\begin{enumerate}[1.]
\item
\label{itemlabel:prop:vanishing:1}
The quotient $S/\mathfrak{S}_\kk$ is 
semi-algebraically
homeomorphic to $\boldPsi^{(\kk)}_\kk(S)$.
\item
\label{itemlabel:prop:vanishing:2}
For any field of coefficients $\F$,
\[
\HH^*(\boldPsi^{(\kk)}_\kk(S),\F) \cong \HH^*(\boldPsi^{(\kk)}_\dd(S),\F).
\]
\end{enumerate}
\end{proposition}

\begin{proof}
Part \eqref{itemlabel:prop:vanishing:1} follows from the fact the map $\boldPsi^{(\kk)}_\kk$ separates orbits of $\mathfrak{S}_\kk$,
and Theorem \ref{thm:Weyl-diffeo}.

In order to prove Part \eqref{itemlabel:prop:vanishing:2}
first note that
\[
\R[\X^{(1)},\ldots,\X^{(\omega)}]^{\mathfrak{S}_\kk}_{\leq \dd} \cong \R[\X^{(1)}]^{\mathfrak{S}_{k_1}}_{\leq d_1} \otimes \cdots \otimes \R[\X^{(\omega)}]^{\mathfrak{S}_{k_\omega}}_{\leq d_\omega},
\]
and for each $i, 1 \leq i \leq \omega$, 
\[
\R[\X^{(i)}]^{\mathfrak{S}_{k_i}} = \R[p_1^{(k_i)}(\X^{(i)}),\ldots, p_{k_i}^{(k_i)}(\X^{(i)})].
\]

It follows that for each $P \in \mathcal{P}$, there exists $\widetilde{P} \in \R[\ZB^{(1)},\ldots,\ZB^{(\omega)}]$, 
with $\ZB^{(i)} = (Z^{(i)}_1,\ldots, Z^{(i)}_{d_i'}), 1 \leq i \leq \omega$, such that
\[
P = \widetilde{P}(p_1^{(k_1)}(\X^{(1)}),\ldots, p_{d_1'}^{(k_1)}(\X^{(1)}),\ldots,p_1^{(k_\omega)}(\X^{(\omega)}),\ldots p_{d_\omega'}^{(k_\omega)}(\X^{(\omega)})).
\]

Let $\widetilde{\mathcal{P}} = \{\widetilde{P} \;\mid\; P \in \mathcal{P}\}$. Also, 
let $\boldsymbol{\Theta}$ be a $\mathcal{P}$-closed formula defining $S$, and 
$\widetilde{\boldsymbol{\Theta}}$ be the $\widetilde{\mathcal{P}}$-closed formula obtained from $\boldsymbol{\Theta}$ by
replacing for each $P \in \mathcal{P}$, every occurrence of $P$ by $\widetilde{P}$.

Now observe that 
\[
\boldPsi_{\dd}^{(\kk)} =  \boldpi_{\kk,\dd} \circ  \boldPsi_{\kk}^{(\kk)},
\]
where

$\boldpi_{\kk,\dd}: \R^k \rightarrow \R^{d_1'} \times \cdots \times \R^{d_\omega'}$ denotes the map
\[
\boldpi_{\kk,\dd}(\x^{(1)},\ldots,\x^{(\omega)}) = (\pi_{k_1,d_1}(\x^{(1)}),\ldots, \pi_{k_\omega,d_\omega}(\x^{(\omega)})),
\]

where for each $i, 1 \leq i \leq \omega$, 
$\pi_{k_i,d_i} (\x^{(i)}) = (x^{(i)}_1,\ldots, x^{(i)}_{d_i'})$.

The quotient space $S/\mathfrak{S}_\kk$ is homeomorphic to 
$\boldPsi_{\kk}^{(\kk)}(S)$, and
\[
\boldPsi_\kk^{(\kk)}(S) = \RR(\widetilde{\boldsymbol{\Theta}},\R^k) \cap \boldPsi_\kk^{(\kk)}(\R^k).
\]

It is also clear from the definition of $\widetilde{\boldsymbol{\Theta}}$, that 
\[
\boldsymbol{\pi}_{\kk,\dd}^{-1}(\boldsymbol{\pi}_{\kk,\dd}(\RR(\widetilde{\boldsymbol{\Theta}},\R^k))) = \RR(\widetilde{\boldsymbol{\Theta}},\R^k)
\]
(in other words $\RR(\widetilde{\boldsymbol{\Theta}},\R^k)$ is equal to the cylinder over 
$\boldsymbol{\pi}_{\kk,\dd}(\RR(\widetilde{\boldsymbol{\Theta}},\R^k))$).
Also notice that
\[
\boldsymbol{\pi}_{\kk,\dd}(\RR(\widetilde{\boldsymbol{\Theta}},\R^k)) = \boldsymbol{\pi}_{\kk,\dd}(S).
\]

It follows from Corollary \ref{cor:Kostov} that for every 
$\y \in \boldsymbol{\pi}_{\kk,\dd}(\RR(\widetilde{\boldsymbol{\Theta}},\R^k))=
\boldsymbol{\pi}_{\kk,\dd}(\boldPsi_\kk^{(\kk)}(S))
$, 
\begin{eqnarray*}
\boldsymbol{\pi}_{\kk,\dd}^{-1}(\y) \cap \boldPsi_\kk^{(\kk)}(\R^k)
\end{eqnarray*}
is contractible.

Now in the case $\R = \mathbb{R}$, the Vietoris-Begle mapping theorem (see for instance, \cite[page 344]{Spanier}) implies 
that 
\[
\HH^*(\boldPsi_\kk^{(\kk)}(S),\F) \cong \HH^*(\boldsymbol{\pi}_{\kk,\dd} \circ \boldsymbol{\psi}_{\kk}^{\kk}(S) ,\F) = \HH^*(\boldPsi_{\dd}^{\kk}(S),\F),
\]
proving
Part \eqref{itemlabel:prop:vanishing:2} in the case $\R = \mathbb{R}$. The general case follows from
an application of the Tarski-Seidenberg transfer principle.
\end{proof}

\begin{proof}[Proof of Theorem \ref{thm:vanishing}]
First note that using  Proposition \ref{prop:vanishing}, 
\[
\HH^*(S/\mathfrak{S}_\kk,\F) \cong  \HH^*(\boldPsi^{(\kk)}_\dd(S),\F),
\] 
and 
$\boldPsi^{(\kk)}_\dd(S)$ is a semi-algebraic subset of $\R^N$, where
$N = \sum_{i=1}^{\omega} \min(k_i , d_i)$. Since no semi-algebraic subset of $\R^N$ can have non-vanishing homology in dimensions $N$ or greater, the theorem follows.
\end{proof}

\begin{remark}[Tightness]
\label{rem:vanishing:tightness}
\hide{
The vanishing result in Theorem \ref{thm:vanishing} is tight as can be seen from the following example.
Since $\Psi^{(k)}_k$ is a homeomorphism on to its image, clearly its image has dimension equal to $k$, and it follows that the image of  
$\Psi^{(k)}_d$ has dimension $d$.
} 
Suppose that $d < k$.
Observe that the image of $\Psi^{(k)}_d$ is a non-empty semi-algebraic subset of
$\R^d$ having dimension $d$, and thus has a non-empty interior. Let $\y = (y_1,\ldots, y_d)$ belong to the interior of the image
of $\Psi^{(k)}_d$. Then, for all small enough $\eps > 0$, the intersection of  the image of $\Psi^{(k)}_d$ with the union of the $2d$ hyperplanes defined by 
\begin{equation}
\label{eqn:rem:vanishing:tightness} 
p^{(k)}_i = y_i \pm \eps, 1\leq i \leq d,
\end{equation}
contains the boundary of the hypercube $[y_1 -\eps, y_1 +\eps] \times \cdots \times [y_d -\eps,y_d +\eps]$ but not its interior, and thus clearly has non-vanishing cohomology in dimension $d-1$.
Using Part \eqref{itemlabel:prop:vanishing:2} of Theorem \ref{prop:vanishing}, it follows that the
symmetric semi-algebraic $S \subset \R^k$ defined by \eqref{eqn:rem:vanishing:tightness} has 
$\HH^{d-1}(S,\F) \neq 0$. Finally note that,  the symmetric polynomials,
\[
p^{(k)}_i - y_i \pm \eps, 1\leq i \leq d,
\]
defining $S$ have degrees bounded by $d$.
\end{remark}

\subsection{Proof of Theorem \ref{thm:bound}}

\begin{notation}
\label{not:composition}
For $k \in \Z_{\geq 0}$, we denote by $\Comp(k)$ the set of integer tuples 
\[
\lambda= (\lambda_1,\ldots,\lambda_\ell), \lambda_i > 0, |\lambda| := \sum_{i=1}^{\ell} \lambda_i = k.
\] 
\end{notation}

\begin{definition}
\label{def:composition-order}
For $k \in \Z_{\geq 0}$, and $\lambda = (\lambda_1,\ldots,\lambda_\ell) \in \Comp(k)$,
we denote by $\Weyl_{\lambda}$ the subset of $\Weyl^{(k)}$ defined by,
\[
X_1 = \cdots = X_{\lambda_1} \leq X_{\lambda_1+1} = \cdots = X_{\lambda_1+\lambda_2} \leq \cdots \leq X_{\lambda_1+\cdots+\lambda_{\ell-1}+1} = \cdots = X_k,
\]
and 
denote by $\Weyl_{\lambda}^o$ the subset of $\Weyl^{(k)}$ defined by,
\[
X_1 = \cdots = X_{\lambda_1} < X_{\lambda_1+1} = \cdots = X_{\lambda_1+\lambda_2} < \cdots < X_{\lambda_1+\cdots+\lambda_{\ell-1}+1} = \cdots = X_k,
\]

More generally, given $\kk = (k_1,\ldots,k_\omega) \in \Z_{\geq 0}$, we denote
\[
\bWeyl^{(\kk)} = \Weyl^{(k_1)} \times \cdots \times \Weyl^{(k_\omega)}.
\]
Given $\boldlambda =(\lambda^{(1)},\ldots,\lambda^{(\omega)}) \in \bComp(\kk,\dd)$ we denote 
\[
\bWeyl_{\boldlambda} = \Weyl_{\lambda^{(1)}} \times \cdots \times \Weyl_{\lambda^{(\omega)}}.
\]
\end{definition}

\begin{definition}
\label{def:poset-compcat}
Let $k\in \Z_{\geq 0}$, and  $\lambda,\mu \in \Comp(k)$. We denote, $\lambda \prec \mu$, if $\Weyl_{\lambda} \subset \Weyl_{\mu}$.
It is clear that $\prec$ is a partial order on $\Comp(k)$ making $\Comp(k)$ into a poset.

For $\kk \in \Z_{\geq 0}^\omega$,
and  $\boldlambda=(\lambda^{(1)},\ldots,\lambda^{(\omega)}), \boldmu = (\mu^{(1)},\ldots,\mu^{(\omega)}) \in \bComp(\kk)$,
we denote, $\boldlambda \prec \boldmu$, if
$\lambda^{(i)} \prec \mu^{(i)}$ for all $i, 1\leq i \leq \omega$.
It its clear that $\prec$ extends the partial order on $\Comp(k)$ defined above.
\end{definition}

\begin{notation}
For $\lambda = (\lambda_1,\ldots,\lambda_\ell) \in \Comp(k)$, we denote
$\length(\lambda) = \ell$,
and
for $k,d \in \Z_{\geq 0}$,
we denote 
\begin{eqnarray*}
\CompMax(k,d)  &=& \{\lambda=(\lambda_1,\ldots,\lambda_d) \in \Comp(k) \; \mid \;
\lambda_{2i+1} =1, 0 \leq i < d/2 \},\\
\Comp(k,d) &=&
\bigcup_{\lambda  \in \CompMax(k,d)} \{\lambda' \in \Comp(k) \mid \lambda' \prec  \lambda\} 
\mbox{ if $d \leq k$}, \\
&=& \Comp(k), \mbox{ if $d > k$}. 
\end{eqnarray*}
More generally, for $\kk, \dd \in \Z_{\geq 0}^\omega$, we denote
\[
\bComp(\kk,\dd) = \Comp(k_1,d_1) \times \cdots \times \Comp(k_\omega,d_\omega).
\]
\end{notation}

\begin{definition}
\label{def:Skd}
Given $k,d \in \Z_{\geq 0}$, 
we denote  
\begin{eqnarray*}
\Weyl^{(k)}_d &=& \bigcup_{\lambda \in \Comp(k,d)} \Weyl_\lambda.
\end{eqnarray*} 
For $k,d \in \Z_{\geq 0}$, and a semi-algebraic subset $S \subset \R^k$, we denote

\begin{equation}
\label{eqn:def:Skd}
S_{k,d} = S \cap \Weyl^{(k)}_d.
\end{equation}
(Notice that if $d \geq k$, then $S_{k,d} = S \cap \Weyl^{(k)}$.)

We will denote by $L_\lambda$ the linear span of $\Weyl_\lambda$. Note that 
\[
\dim L_\lambda = \dim \Weyl_\lambda = \length(\lambda).
\]

More generally,
given $\dd =(d_1,\ldots,d_\omega), \kk=(k_1,\ldots,k_\omega)  \in \Z_{\geq 0}^\omega$
with $k = |\kk|$, 
we denote  
\begin{eqnarray*}
\bWeyl^{(k)}_\dd &=& \Weyl^{(k_1)}_{d_1} \times \cdots \times \Weyl^{(k_\omega)}_{d_\omega}.
\end{eqnarray*} 
For any semi-algebraic subset $S \subset \R^k$, we denote
\[
S_{\kk,\dd} = S \cap \Weyl^{(\kk)}_\dd.
\]

We will denote by
$L_{\boldlambda}$ the linear span of $\bWeyl_{\boldlambda}$. Note that 

\[
\dim L_{\boldlambda} = \dim \bWeyl_{\boldlambda} = \sum_{i=1}^{\omega}\length(\lambda^{(i)}).
\]

\end{definition}

We will use the following theorem due to Arnold \cite{Arnold}.

\begin{theorem}\label{thm:arnold}\cite[Theorem 7]{Arnold}
\begin{enumerate}[1.]
\item
\label{itemlabel:thm:arnold:a}
For every $\w \in \R_{\geq 0}^k$, $d,k \geq 0$, 
$d'  = \min(k,d)$,
and $\y \in \R^{d'}$
the function $p_{\w,d+1}^{(k)}$ has exactly one local maximum on
$(\Psi^{(k)}_{\w,d})^{-1}(\y)$,
which furthermore depends continuously on $\y$.
\item
\label{itemlabel:thm:arnold:b}
Suppose that the real variety $V_\y \subset \R^k$ defined by $(p^{(k)}_1,\ldots,p^{(k)}_{d'}) = \y$ is non-singular. Then a point $\x \in V_\y \cap \Weyl^{(k)}$ is a local maximum if and only if $\x \in \Weyl_\lambda^{(k)}$ for some $\lambda \in \CompMax(k,d')$.
\end{enumerate}  
\end{theorem}

\begin{remark}
Note that in \cite[Theorem 7]{Arnold} there is a slight inaccuracy in that  the word ``minimum" should be replaced by the word ``maximum"
and vice versa. A correct statement and a more detailed proof can be found in \cite{Meguerditchian} (Proposition 8).
\end{remark}

Let $d >1$, and 
for $\y\in \Psi^{(k)}_{d}(\Weyl^{(k)})$,  let 
\[
m(\y):=\min_{\x \in (\Psi^{(k)}_{d})^{-1}(\y)} p^{(k)}_{d+1}(\x).
\]

By Part \eqref{itemlabel:thm:arnold:a} of Theorem \ref{thm:arnold} the map,
$F^{(k)}_d:  \Psi^{(k)}_{d}(\Weyl^{(k)}) \rightarrow \Weyl^{(k)}$ which sends $\y \in \Psi^{(k)}_{d}(\Weyl^{(k)})$ to the unique $\x \in \Weyl^{(k)}$,
such that $m(\y) = p^{(k)}_{d+1}(\x)$  
is a well-defined semi-algebraic continuous map.

Let 
$U^{(k)}_d \subset \Psi^{(k)}_{d}(\Weyl^{(k)})$ be the subset of points $\y=(y_1,\ldots,y_{d'})$ of 
$\Psi^{(k)}_{d}(\Weyl^{(k)})$ such that $V_\y \subset \R^k$ defined by $(p^{(k)}_1,\ldots,p^{(k)}_{d'}) = \y$ is non-singular. 

We have the following equalities.
\begin{proposition}
\label{prop:arnold-preparation}
\[
\Weyl^{(k)}_d = \overline{F^{(k)}_d(U^{(k)}_d) } = F^{(k)}_d(\Psi^{(k)}_{d}(\Weyl^{(k)})).
\]
\end{proposition}

\begin{proof}
The second equality follows from the continuity of $F^{(k)}_d$, and the fact that 
by semi-algebraic version of Sard's theorem (see for example \cite[Chapter 5]{BPRbook2}), $U^{(k)}_d$ is
dense in $\Psi^{(k)}_{d}(\Weyl^{(k)})$.

We now prove the first equality.

The inclusion $\overline{F^{(k)}_d(U^{(k)}_d) } \subset \Weyl^{(k)}_d$ is clear, since by 
Part \eqref{itemlabel:thm:arnold:b} of Theorem \ref{thm:arnold},
$F^{(k)}_d(U^{(k)}_d) \subset \Weyl^{(k)}_d$, and $\Weyl^{(k)}_d$ is closed.

We now prove the inclusion $\Weyl^{(k)}_d \subset \overline{F^{(k)}_d(U^{(k)}_d) }$.
Let $\x \in \Weyl^{(k)}_d$.  Then there exists $\lambda \in \CompMax(k,d)$ such that $\x \in \Weyl_\lambda$.
The map $\Psi^{(k)}_d$ is a local diffeomorphism on $\Weyl_\lambda^o$, and the set dimension of the set of critical values
of  $\Psi^{(k)}_d$ is of dimension at most $d-1$ by the semi-algebraic version of Sard's theorem. Thus, there exists 
$\x' \in \Ext(\Weyl_\lambda^o,\R\la\eps\ra)$ such that $\lim_\eps \x' = \x$, 
$\y' = \Psi^{(k)}_d(\x')$
is a regular value of the map $\Psi^{(k)}_d$, and hence
$\y'  \in \Ext(U^{(k)}_d,\R\la\eps\ra)$.
Then, $\x' = F^{(k)}_d(\y')  \in \Ext(F^{(k)}_d(U^{(k)}_d),\R\la\eps\ra)$, and since $\x = \lim_\eps \x'$, 
$\x \in \overline{F^{(k)}_d(U^{(k)}_d)}$.
\end{proof}

\begin{proposition}\label{prop:arnold}
Let $1 < d$, and 
$S \subset \R^k$  a closed and bounded symmetric semi-algebraic set defined by symmetric polynomials of degrees bounded by $d$.
Then the  following holds.
\begin{enumerate}[1.]
\item
\label{itemlabel:prop:arnold:b}
The map $\Psi^{(k)}_{d}$ restricted to $S_{k,d}$ is a semi-algebraic homeomorphism on to its image.
\item
\label{itemlabel:prop:arnold:c}
\[
\HH^*(S_{k,d},\F) \cong  \HH^*(S/\mathfrak{S}_k,\F).
\]
\end{enumerate}

More generally, let 
$\dd,\kk \in \Z_{> 1}^\omega$ with $1^\omega < \dd$,
and $S$ a 
bounded
$\mathcal{P}$-closed semi-algebraic set, where
$\mathcal{P} \subset \R[\X^{(1)},\ldots,\X^{(\omega)}]^{\mathfrak{S}_\kk}_{\leq \dd}$.
Then,

\begin{enumerate}[1{$'$}.]
\item
$\boldPsi^{(\kk)}_{\dd}$ restricted to $S_{\kk,\dd}$ is a semi-algebraic homeomorphism on to its image, and
\item
\[
\HH^*(S_{\kk,\dd},\F) \cong  \HH^*(S/\mathfrak{S}_\kk,\F).
\]
\end{enumerate}
\end{proposition}
\begin{proof}
We only prove Parts  \eqref{itemlabel:prop:arnold:b} and
\eqref{itemlabel:prop:arnold:c}. The remaining parts follow directly from these two.
Part \eqref{itemlabel:prop:arnold:b} follows from Proposition \ref{prop:arnold-preparation},
and
Part \eqref{itemlabel:prop:arnold:c} follows  from Part \eqref{itemlabel:prop:arnold:b} and
Proposition \ref{prop:vanishing}.
\end{proof}

\begin{example}
\label{eg:3-sphere}
In order to understand the geometry behind Proposition \ref{prop:arnold} it might be useful to
consider the example of the two-dimensional sphere in $S \subset\R^3$ defined by the symmetric quadratic polynomial equation
\[
p_2^{(3)}(X_!,X_2,X_3) - 1= X_1^2 + X_2^2 + X_3^2 -1 =0.
\]
 
The intersection of $S$ with the Weyl chamber, $\Weyl^{(3)}$ defined
by $X_1 \leq X_2 \leq X_3$, is contractible and is homologically equivalent to $S / \mathfrak{S}_3$, via
the map $\Psi^{(3)}_2 = (p^{(3)}_1, p^{(3)}_2): S \cap \Weyl^{(3)} \rightarrow \R^2$. The image of this map
in $\R^2$ is the line segment defined by $-\sqrt{3} \leq p^{(3)}_1 \leq \sqrt{3}, p_2^{(3)} = 1$, and
is homotopy
equivalent to $S / \mathfrak{S}_3$. For each $\y = (y_1,y_2) \in \R^2$ which belongs to the image, the fiber 
$(\Psi^{(3)}_2)^{-1}(\y) \subset S$ is defined by 
\[
X_1+X_2+X_3 = y_1, X_1^2 + X_2^2 + X_3^2 =1, X_1 \leq X_2 \leq X_3,
\] 
and is 
easily seen to be a connected arc and hence contractible. Moreover, the maximum of $p^{(3)}_3$ restricted to this arc belong to the face defined by $X_2 = X_3$ of the Weyl chamber. The set, $S_{3,2}$ of these maximums, is an arc defined by  
\[
X_1^2 + X_2^2 + X_3^2 =1, X_1\leq  X_2 = X_3,
\]
and 
defines a section over the image of $\Psi^{(3)}_2(S \cap \Weyl^{(3)})$, and is homologically equivalent to  to $S/\mathfrak{S}_3$.
Notice  also that $S_{3,2}$ is contained in the face $\Weyl^{(3)}_\lambda$, where $\lambda = (1,2) \in \Comp(k,2)$.
The  two sets, $S\cap  \Weyl^{(3)}$ and  $S_{3,2}$, are shown in Figure \ref{fig:sphere}.

\begin{figure}
\subfigure[$S\cap  \Weyl^{(3)}$]{\includegraphics[width=0.49\textwidth]{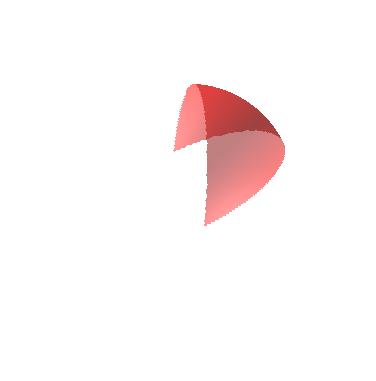}}
    \subfigure[$S_{3,2}$]{\includegraphics[width=0.49\textwidth]{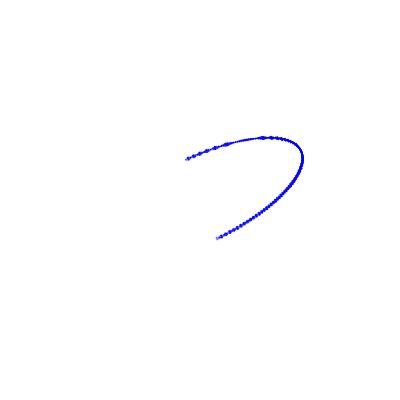}}
\caption{Visualization of Example \ref{eg:3-sphere}.}
\label{fig:sphere}
\end{figure}

\end{example}

The following is easy to prove.

\begin{proposition}
\label{prop:closure-compcat}
Let $\lambda, \lambda' \in \Comp(k,d)$. Then there exists $\lambda'' \in \Comp(k,d)$ such that 
$\Weyl_{\lambda''} = \Weyl_{\lambda}\cap \Weyl_{\lambda'}$. 

More generally, let $\kk,\dd \in \Z_{\geq 0}^\omega$, and
let $\boldlambda, \boldlambda' \in \bComp(\kk,\dd)$. Then there exists $\boldlambda'' \in \bComp(\kk,\dd)$ such that 
$\bWeyl_{\boldlambda''} = \bWeyl_{\boldlambda}\cap \bWeyl_{\boldlambda'}$. 

\end{proposition}

\hide{
\begin{definition}
\label{def:poset-compcat}
Let $k,d \in \Z_{\geq 0}$, and  $\lambda,\mu \in \Comp(k,d)$. We denote, $\lambda \prec \mu$, if $\Weyl_{\lambda} \subset \Weyl_{\mu}$.
It is clear that $\prec$ is a partial order on $\Comp(k,d)$ making $\Comp(k,d)$ into a poset.

For $\kk,\dd \in \Z_{\geq 0}^\omega$,
and  $\boldlambda=(\lambda^{(1)},\ldots,\lambda^{(\omega)}), \boldmu = (\mu^{(1)},\ldots,\mu^{(\omega)}) \in \bComp(\kk,\dd)$,
we denote, $\boldlambda \prec \boldmu$, if
$\lambda^{(i)} \prec \mu^{(i)}$ for all $i, 1\leq i \leq \omega$.
It its clear that $\prec$ extends the partial order on $\Comp(k,d)$ defined above.
\end{definition}
}

Recall that a chain $\sigma$ of a finite poset $P$ is an ordered sequence 
$\sigma_1 \prec \sigma_2 \prec \cdots \prec \sigma_m$ with
$\sigma_i \neq \sigma_{i+1}$ for  $1 \leq i < m$.

\begin{notation}
For $d,k \geq 0$, we denote by $\Sigma_{k,d} $ denote the set of chains of the poset $\Comp(k,d)$.
More generally, for $\kk,\dd \in \Z_{\geq 0}^\omega$, we denote by
$\boldSigma_{\kk,\dd}$ the chains of the poset $\bComp(\kk,\dd)$.
\end{notation}

\begin{proposition}
\label{prop:chainsCompcat}
 For $d,k \geq 0$,
\begin{eqnarray*}
\card(\Sigma_{k,d}) &\leq& 
(2^d -1) \prod_{i=1}^{\lfloor d/2 \rfloor -1}(k - \lceil d/2 \rceil-i) \mbox{ if $d \leq k$}, \\
&\leq& (2^k -1) (k-1)! \mbox{ if $d > k$}.
\end{eqnarray*} 

More generally, for $\dd = (d_1,\ldots,d_\omega),\kk=(k_1,\ldots,k_\omega) \in \Z_{\geq 0}^\omega$,
\begin{eqnarray*}
\card(\boldSigma_{\kk,\dd}) &=& \prod_{i=1}^{\omega} \card(\Sigma_{k_i,d_i}).
\end{eqnarray*}
\end{proposition}

\begin{proof}
It is easy to see that the number of maximal chains (of length $d$ in $\Comp(k,d)$)  is
equal to 
\[
\prod_{i=1}^{\lfloor d/2\rfloor - 1}(k - \lceil d/2 \rceil - i).
\] 
Each maximal chain has  $(2^d -1)$ sub-chains.
Some of these chains are being counted more than once, but we are only interested in an upper bound.
\end{proof}

\subsubsection{Systems of neighborhoods}

Let $\overline{\eps} = (\eps_0,\ldots,\eps_k)$, and for $0 \leq i \leq k$,
$\R_i = \R\la \eps_0,\ldots,\eps_i \ra$.

\begin{definition}
\label{def:tilde-W-lambda}
For $k,d \in \Z_{\geq 0}$, $\lambda \in \Comp(k,d)$, we denote
\begin{eqnarray*}
P_\lambda &=& \sum_{i=1}^{\length(\lambda)} \sum_{j=\lambda_1+\cdots+\lambda_{i-1}+1}^{\lambda_1+\cdots+\lambda_i}\sum_{j'=j+1}^{\lambda_1+\cdots+\lambda_i} (X_j - X_{j'})^2,
\end{eqnarray*}
and
\begin{equation*}
\widetilde{\Weyl}_\lambda =
\{ \x \in \Ext(\Weyl^{(k)},\R_{\length(\lambda)}) \mid 
(P_\lambda - \eps_{\length(\lambda)} \leq 0) \wedge \bigwedge_{\substack{\mu \prec \lambda,\\
\mu \neq \lambda}} (P_\mu - \eps_{\length(\mu)} \geq 0) \}.
\end{equation*}

More generally,
for $\kk,\dd \in \Z_{\geq 0}^\omega$, and $\boldlambda = (\lambda^{(1)},\ldots,\lambda^{(\omega)}) \in \bComp(\kk,\dd)$, we denote
\begin{eqnarray*}
P_{\boldlambda}  &=& \sum_{i=1}^{\omega} P_{\lambda^{(i)}},
\end{eqnarray*}
and
\begin{equation*}
\widetilde{\bWeyl}_{\boldlambda} =
\{ \x \in \Ext(\Weyl^{(\kk)},\R_{\length(\boldlambda)}) \mid 
(P_{\boldlambda}  - \eps_{\length(\boldlambda)} \leq 0) \wedge \bigwedge_{\substack{\boldmu \prec \boldlambda,\\
\boldmu \neq \boldlambda}} (P_{\boldmu} - \eps_{\length(\boldmu)} \geq 0) \}.
\end{equation*}
\end{definition}

\begin{example}
\begin{figure}
\begin{picture}(0,0)%
\includegraphics{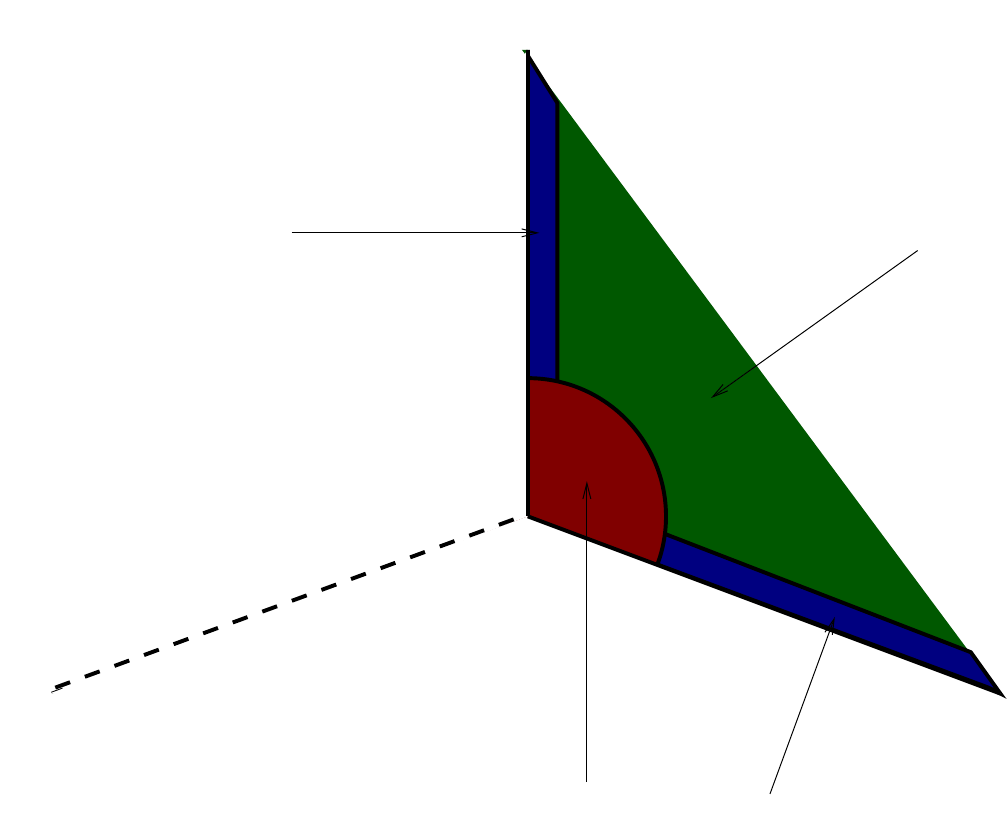}%
\end{picture}%
\setlength{\unitlength}{2486sp}%
\begingroup\makeatletter\ifx\SetFigFont\undefined%
\gdef\SetFigFont#1#2#3#4#5{%
  \reset@font\fontsize{#1}{#2pt}%
  \fontfamily{#3}\fontseries{#4}\fontshape{#5}%
  \selectfont}%
\fi\endgroup%
\begin{picture}(7653,6279)(931,-6475)
\put(2431,-1996){\makebox(0,0)[lb]{\smash{{\SetFigFont{7}{8.4}{\familydefault}{\mddefault}{\updefault}{\color[rgb]{0,0,0}$\widetilde{\Weyl}_{(2,1)}$}%
}}}}
\put(4681,-331){\makebox(0,0)[lb]{\smash{{\SetFigFont{7}{8.4}{\familydefault}{\mddefault}{\updefault}{\color[rgb]{0,0,0}$X_1=X_2$}%
}}}}
\put(8326,-5776){\makebox(0,0)[lb]{\smash{{\SetFigFont{7}{8.4}{\familydefault}{\mddefault}{\updefault}{\color[rgb]{0,0,0}$X_2=X_3$}%
}}}}
\put(946,-5731){\makebox(0,0)[lb]{\smash{{\SetFigFont{7}{8.4}{\familydefault}{\mddefault}{\updefault}{\color[rgb]{0,0,0}$X_1=X_3$}%
}}}}
\put(6571,-6406){\makebox(0,0)[lb]{\smash{{\SetFigFont{7}{8.4}{\familydefault}{\mddefault}{\updefault}{\color[rgb]{0,0,0}$\widetilde{\Weyl}_{(1,2)}$}%
}}}}
\put(8011,-2041){\makebox(0,0)[lb]{\smash{{\SetFigFont{7}{8.4}{\familydefault}{\mddefault}{\updefault}{\color[rgb]{0,0,0}$\widetilde{\Weyl}_{(1,1,1)}$}%
}}}}
\put(5401,-6361){\makebox(0,0)[lb]{\smash{{\SetFigFont{7}{8.4}{\familydefault}{\mddefault}{\updefault}{\color[rgb]{0,0,0}$\widetilde{\Weyl}_{(3)}$}%
}}}}
\end{picture}%
\caption{Cross-section of $\widetilde{W}^{(3)}$ in the hyperplane $X_1+X_2+X_3=0$.} 
\label{fig:Weyl}
\end{figure}
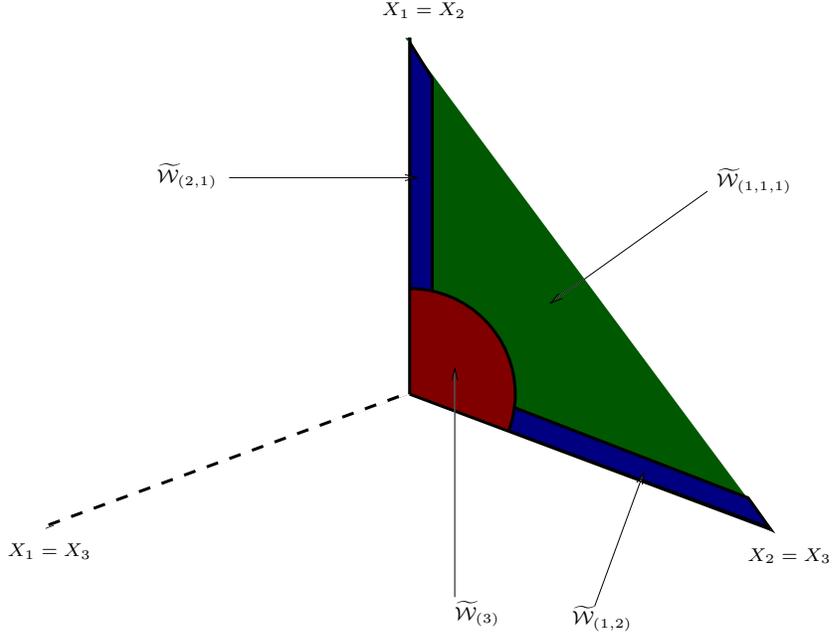
Before proceeding further it might be useful to visualize the different $\widetilde{\Weyl}_{\lambda}$ 
in the case $k=3$.
We display the intersections of different  $\widetilde{\Weyl}_\lambda, \lambda \in \Comp(3)$
with the hyperplane defined by $X_1 + X_2 + X_3 = 0$ in Figure \ref{fig:Weyl}. The Hasse diagram of the poset $\Comp(3)$ is
as follows.

\[
\xymatrix{
& (1,1,1) &\\
(1,2) \ar[ur] && (2,1) \ar[ul] \\
& (3) \ar[ul] \ar[ur] &
}
\]

It is clear from the Figure \ref{fig:Weyl}, that for $\Lambda \subset \Comp(3)$, 
\[
\bigcap_{\lambda \in \Lambda} \widetilde{\Weyl}_\lambda
\] 
is non-empty  and only if the elements of $\Lambda$ form a chain in $\Comp(3)$. The list of chains in $\Comp(3)$ is
$$\displaylines{
(3), (1,2), (2,1), (1,1,1), \cr
(3) \prec (1,2), (3) \prec (2,1), (3) \prec (1,1,1), (1,2) \prec(1,1,1), (2,1) \prec (1,1,1), \cr
(3)\prec (1,2) \prec (1,1,1), (3) \prec (2,1) \prec (1,1,1).
}
$$
It can be seen from Figure \ref{fig:Weyl} that the corresponding intersections of the $\widetilde{\Weyl}_\lambda$'s for each chain listed above is non-empty.
\end{example}

\begin{notation}
For $k,d \in \Z_{\geq 0}$, $\lambda \in \Comp(k,d)$, and
for any semi-algebraic subset $S \subset \R^k$, we denote by $\widetilde{S}_\lambda$ the set
$\Ext(S,\R_{\length(\lambda)}) \cap \widetilde{\Weyl}_\lambda$, and denote
\[
\widetilde{S}_{k,d} = \bigcup_{\lambda \in \Comp(k,d)} \Ext(\widetilde{S}_\lambda,\R_{d'}),
\] 
where $d' = \min(k,d)$.

For a chain $\sigma \in \Sigma_{k,d}$, we denote  
\[
\widetilde{S}_\sigma  =  \bigcap_{\lambda \in \sigma} \Ext(\widetilde{S}_\lambda,\R_\ell),
\]
where $\ell = \length(\max(\sigma))$.

More generally,
for $\kk=(k_1,\ldots,k_\omega),\dd=(d_1,\ldots,d_\omega) \in \Z_{\geq 0}^\omega$, $k = |\kk|$, $\boldlambda \in \bComp(\kk,\dd)$,
and for any semi-algebraic subset $S \subset \R^k$, we denote by $\widetilde{S}_{\boldlambda}$ the set
$\Ext(S,\R_{\length(\boldlambda)}) \cap \widetilde{\bWeyl}_{\boldlambda}$, and denote
\[
\widetilde{S}_{\kk,\dd} = \bigcup_{\boldlambda \in \bComp(\kk,\dd)} \Ext(\widetilde{S}_{\boldlambda},\R_d'),
\] 
where $d' = \sum_{i=1}^{\omega} \min(k_i,d_i)$.
For a chain $\boldsigma \in \boldSigma_{k,d}$, we denote  
\[
\widetilde{S}_{\boldsigma}  =  \bigcap_{\boldlambda \in \boldsigma} \Ext(\widetilde{S}_{\boldlambda},\R_\ell),
\]
where $\ell = \length(\max(\boldsigma))$.
\end{notation}

\begin{proposition}
\label{prop:limit}
Let $k,d \in \Z_{\geq 0}$, and $S \subset \R^k$  a closed and bounded semi-algebraic set. Then,
\[
\lim_{\eps_0} \widetilde{S}_{k,d} = S_{k,d}.
\]
More generally,
let $\kk,\dd \in \Z_{\geq 0}^\omega$, $k = |\kk|$, and  $S \subset \R^k$ a closed and bounded semi-algebraic set. Then,
\[
\lim_{\eps_0} \widetilde{S}_{\kk,\dd} = S_{\kk,\dd}.
\]

\end{proposition}

\begin{proof}
Use Lemma 16.17 in \cite{BPRbook2}.
\end{proof}

\begin{proposition}
\label{prop:empty-intersection}
\begin{enumerate}
\item
\label{itemlabel:prop:empty-intersection:1}
Let $k,d \in \Z_{\geq 0}$, and $\lambda,\mu \in \Comp(k,d)$ such that $\lambda \not\prec \mu, \mu \not\prec \lambda$. Then, 
\[
\Ext(\widetilde{\Weyl}_\lambda, \R_{\ell})  \cap \Ext(\widetilde{\Weyl}_{\mu},\R_{\ell}) = \emptyset,
\]
where $\ell = \max(\length(\lambda),\length(\mu))$.
\item
\label{itemlabel:prop:empty-intersection:2}
More generally,
let $\kk,\dd \in \Z_{\geq 0}^\omega$, and $\boldlambda,\boldmu \in \bComp(\kk,\dd)$ such that $\boldlambda \not\prec \boldmu, \boldmu \not\prec \boldlambda$. Then, 
\[
\Ext(\widetilde{\bWeyl}_{\boldlambda}, \R_{\ell})  \cap \Ext(\widetilde{\bWeyl}_{\boldmu},\R_{\ell}) = \emptyset,
\]
where $\ell = \max(\length(\boldlambda),\length(\boldmu))$.
\end{enumerate}
\end{proposition}

\begin{proof}
We first prove Part \eqref{itemlabel:prop:empty-intersection:1}.
Suppose that
\[
\Ext(\widetilde{\Weyl}_\lambda, \R_{\ell})  \cap \Ext(\widetilde{\Weyl}_{\mu},\R_{\ell}) \neq \emptyset,
\]
and $\x \in \Ext(\widetilde{\Weyl}_\lambda, \R_{\ell})  \cap \Ext(\widetilde{\Weyl}_{\mu},\R_{\ell})$. This implies, using Definition
\ref{def:tilde-W-lambda}  that

\begin{eqnarray}
\label{eqn:prop:empty-intersection:1}
P_\nu(\x) &\geq& \eps_{\length(\nu)},
\end{eqnarray}
where $\nu \in \Comp(k,d)$ is characterized by  $\Weyl_\nu = \Weyl_\lambda \cap \Weyl_\mu$.

Note that, since $\lambda,\mu$ are not comparable by hypothesis, $\nu\neq\lambda,\mu$,
and hence 
\begin{eqnarray}
\label{eqn:prop:empty-intersection:2}
\ell &>& \length(\nu).
\end{eqnarray}
Let $\y \in \lim_{\eps_\ell} \x$.
Then, $\y \in \Weyl_\lambda \cap \Weyl_\mu =\Weyl_\nu$, and hence

\begin{eqnarray}
\label{eqn:prop:empty-intersection:3}
P_\nu(\y) &=& 0.
\end{eqnarray}

On the other hand,
\begin{eqnarray*}
P_\mu(\y) &=& P_\mu(\lim_{\eps_\ell}(\x))\\
&=& \lim_{\eps_\ell} P_\mu(\x) \\
&=& \lim_{\eps_\ell} \eps_{\length(\mu)} \mbox{ (using \eqref{eqn:prop:empty-intersection:1})} \\
&\neq& 0 \mbox{ (since $\ell > \length(\mu)$ by \eqref{eqn:prop:empty-intersection:2},  which implies that
$\eps_{\length(\mu)} \gg \eps_\ell$)}.
\end{eqnarray*}
 This contradicts \eqref{eqn:prop:empty-intersection:3}, which finishes the proof.
 
 Part \eqref{itemlabel:prop:empty-intersection:2} follows immediately from
 Part \eqref{itemlabel:prop:empty-intersection:1} and the definition of the partial order on
 $\bComp(\kk,\dd)$ resulting from the restriction of the one on $\bComp(\kk)$ (cf. Definition  \ref{def:poset-compcat}).
\end{proof}

\begin{corollary}
\label{cor:chains}
Let $k,d \in \Z_{\geq 0}$, and $\Lambda  \subset \Comp(k,d)$. Then 
\[
\bigcap_{\lambda \in \Lambda} \widetilde{\Weyl}_\lambda \neq \emptyset
\] 
only if the elements
of $\Lambda$ form a chain in $\Comp(k,d)$. 

More generally,
let $\kk,\dd \in \Z_{\geq 0}^\omega$, and $\boldLambda  \subset \bComp(\kk,\dd)$. Then 
\[
\bigcap_{\boldlambda \in \boldLambda} \widetilde{\bWeyl}_{\boldlambda} \neq \emptyset
\] 
only if the elements
of $\boldLambda$ form a chain in $\bComp(\kk,\dd)$. 

\end{corollary}

\begin{proof}
Immediate from Proposition \ref{prop:empty-intersection}.
\end{proof}

\begin{proposition}
\label{prop:homotopic}
Let $k,d \in \Z_{\geq 0}$, $\sigma \in \Sigma_{k,d}$  a non-empty chain, and $S \subset \R^k$ a closed and bounded semi-algebraic set. Let $\lambda = \max(\sigma)$, and $\ell = \length(\lambda)$.
Then, 
for any field of coefficients $\F$,
\[
\HH^*(\Ext(L_\lambda,\R_\ell)  \cap \widetilde{S}_\sigma,\F) \cong 
\HH^*(\widetilde{S}_\sigma,\F).
\]

More generally,
let $\kk,\dd \in \Z_{\geq 0}^\omega$, $k = |\kk|$, $\boldsigma \in \boldSigma_{\kk,\dd}$  a non-empty chain, and $S \subset \R^k$ a closed and bounded semi-algebraic set. Let $\boldlambda = \max(\boldsigma)$, and $\ell = \length(\boldlambda)$.
Then, 
for any field of coefficients $\F$,
\[
\HH^*(\Ext(L_{\boldlambda},\R_\ell)  \cap \widetilde{S}_{\boldsigma},\F) \cong 
\HH^*(\widetilde{S}_{\boldsigma},\F).
\]
\end{proposition}

\begin{proof}
Use Lemma 16.17 in \cite{BPRbook2}.
\end{proof}

\begin{proposition}
\label{prop:MV}
\begin{enumerate}[1.]
\item
\label{itemlabel:prop:MV:1}
Let $k,d \in \Z_{\geq 0}, d> 1$, and $S$ a symmetric, $\mathcal{P}$-closed, and bounded semi-algebraic subset of $\R^k$, 
where $\mathcal{P} \subset \R[X_1,\ldots,X_k]^{\mathfrak{S}_k}_{\leq d}$.
Then, 
\begin{eqnarray*}
b(S/\mathfrak{S}_k,\F) \leq \sum_{\sigma \in \Sigma_{k,d}} b(\widetilde{S}_\sigma,\F).
\end{eqnarray*}

\item
\label{itemlabel:prop:MV:2}
More generally, let $\kk,\dd \in \Z_{\geq 0}^\omega$, $k = |\kk|$, and $S$ a symmetric, $\mathcal{P}$-closed, and bounded semi-algebraic subset of $\R^k$, 
where $\mathcal{P} \subset \R[\X^{(1)},\ldots,\X^{(k_\omega)}]^{\mathfrak{S}_\kk}_{\leq \dd}$.
Then, 
\begin{eqnarray*}
b(S/\mathfrak{S}_\kk,\F) \leq \sum_{\boldsigma \in \boldSigma_{\kk,\dd}} b(\widetilde{S}_{\boldsigma},\F).
\end{eqnarray*}
\end{enumerate}
\end{proposition}

\begin{proof}
\noindent
Proof of Part \eqref{itemlabel:prop:MV:1}:
It follows from Part \eqref{itemlabel:prop:arnold:c} of Proposition \ref{prop:arnold} and  Proposition \ref{prop:limit}, that
\[
\HH^*(\Ext(S,\R_d)/\mathfrak{S}_k,\F) \cong \HH^*(\widetilde{S}_{k,d},\F).
\]
Now, 
\[
\widetilde{S}_{k,d}= \bigcup_{\lambda \in \Comp(k,d)} \widetilde{S}_\lambda.
\]

It follows from Part   \eqref{itemlabel:prop:prop1:1} of Proposition \ref{prop:prop1} (Mayer-Vietoris inequality) and Corollary  \ref{cor:chains} that for every $m, 0 \leq m < d$,
\begin{eqnarray*}
b^m(\widetilde{S}_{k,d},\F) &\leq & \sum_{p=0}^{m} \sum_{\substack{\sigma \in \Sigma_{k,d},\\ \card(\sigma) = p+1}} b^{m-p}(\widetilde{S}_\sigma,\F).
\end{eqnarray*} 

Part \eqref{itemlabel:prop:MV:1} of Proposition follows by taking a sum over all $m,0\leq m < d$.

The proof of Part \eqref{itemlabel:prop:MV:2} is similar and omitted.
\end{proof}

 \begin{proof}[Proof of Theorem \ref{thm:bound}]
 Suppose that $S$ is defined by a $\mathcal{P}$-closed formula $\Phi$. 
 We first replace $\R$ by $\R' = \R\la\eps_0\ra$,  and replace $S$ by the $\mathcal{P}'$-closed  
 semi-algebraic set $S'$ defined by the $\mathcal{P}'$-closed formula
 \[
 \Phi \wedge (\eps_0 ||\X||^2 -1 \leq 0).
 \]
  
 Then, using the conical structure theorem for semi-algebraic sets \cite{BPRbook2},  we have
 that,
 \begin{enumerate}[(i)]
 \item
 $S'$ is symmetric, closed and bounded over $\R'$;
 \item
 \begin{eqnarray}
 \label{eqn:thm:bound:S'}
 b^i(S/\mathfrak{S}_k,\F) &=& b^i(S'/\mathfrak{S}_k,\F).
 \end{eqnarray}
 \end{enumerate}
 
 We now obtain an upper bound  $b(\widetilde{S}'_\sigma,\F)$ for each chain $\sigma \in \Sigma_{k,d}$ as follows.
 Using Proposition \ref{prop:homotopic} we have that
 \[
 b(\widetilde{S}'_\sigma,\F) = b(\Ext(L_\lambda,\R'_\ell)  \cap \widetilde{S}'_\sigma,\F),
 \]
 where $\lambda = \max(\sigma)$ and $\ell = \length(\lambda)$.
 Notice that $\widetilde{S}'_\sigma$ is the intersection of the $\mathcal{P}'$-closed semi-algebraic set $S'$, with  the basic closed semi-algebraic set,
defined by 
 \begin{eqnarray}
 \nonumber
P_\mu - \eps_{\length(\mu)}  &=& 0, \mbox{ for $\mu \in \sigma, \mu \neq \lambda$}, \\
\label{eqn:number-of-distinct}
P_\nu - \eps_{\length(\nu)}  &\leq& 0, \nu \not\in \sigma, \nu \prec \lambda.
  \end{eqnarray}
 
 Let
 \[
 \mathcal{F}_\sigma = \bigcup_{\mu \in \sigma, \mu \neq \lambda} \{P_\mu - \eps_{\length(\mu)}\}, \;
 \mathcal{G}_\sigma =  \bigcup_{\nu \not\in \sigma, \nu \prec \lambda}  \{P_\nu - \eps_{\length(\nu)} \}.
 \]
 
Using Corollary \ref{cor:chains}, the number of distinct subsets $\mathcal{G}_\sigma'\subset \mathcal{G}_\sigma$, such that 
 \[
 \ZZ(\mathcal{F}_\sigma \cup \mathcal{G}_\sigma',\R'_\ell) \cap \Ext(\Weyl^{(k)},\R'_\ell) \neq \emptyset
 \] 
 is bounded by 
 \begin{equation}
 \label{eqn:distinct-nonempty}
 (O(d'))^{d'}.
 \end{equation}

We obtain using Proposition \ref{prop:P-closed-main} that
\begin{eqnarray*}
b(\widetilde{S}'_\sigma,\F) &\leq & 
s+\sum_{ p \geq 0} \sum_{\substack{ I \subset [1,s],\\ 1 \leq \card(I) \leq k-p, \\
     J \subset I,\\ 1\leq \card(J) \leq p+1}} \sum_{\tau \in \{0,\pm 1,\pm 2\}^J} 
     \sum_{\mathcal{G}_\sigma'\subset \mathcal{G}_\sigma}  G(p,\card(I),J,K,\tau),
     \end{eqnarray*}
     where
     \[
     G(p,q,J,K,\tau)
     =
     b^{p+q -\card(J)} (\Ext(L_\lambda,\R'_\ell) \cap \ZZ(\mathcal{P}_\tau \cup \mathcal{F}_\sigma \cup \mathcal{G}_\sigma',\R'_\ell) \cap \Ext(\Weyl^{(k)},\R'_\ell),\F),
     \]
and $\mathcal{P}_\tau$ is as in \eqref{eqn:P-tau}.

Since $\dim(L_\lambda) = \length(\lambda) \leq d'$, we obtain using
\eqref{eqn:distinct-nonempty}, Proposition \ref{prop:inductive},
and
Corollary \ref{cor:betti-bound-sa-general},
that,
 \begin{eqnarray}
  \label{eqn:bound:1} 
b(\widetilde{S}_\sigma',\F)  &\leq &  (O(s d d'))^{d'}.
  \end{eqnarray}
  The theorem now follows from \eqref{eqn:thm:bound:S'}, 
 Propositions \ref{prop:chainsCompcat},  \ref{prop:MV}, and
  \eqref{eqn:bound:1}.
 
 \end{proof}
 
\subsection{Proof of Theorem \ref{thm:bound-general}}
\begin{proof}[Proof of Theorem \ref{thm:bound-general}]
The proof is very similar to that of Theorem \ref{thm:bound}.
Suppose that $S$ is defined by a $\mathcal{P}$-closed formula $\Phi$. 
 We first replace $\R$ by $\R' = \R\la\eps_0\ra$,  and replace $S$ by the $\mathcal{P}'$-closed  
 semi-algebraic set $S'$ defined by the $\mathcal{P}'$-closed formula
 \[
 \Phi \wedge (\eps_0 ||\X||^2 -1 \leq 0).
 \]
  
 Then, using the conical structure theorem for semi-algebraic sets \cite{BPRbook2},  we have
 that,
 \begin{enumerate}[i]
 \item
 $S'$ is symmetric, closed and bounded over $\R'$;
 \item
 \begin{eqnarray}
 \label{eqn:thm:bound:S'-general}
 b^i(S/\mathfrak{S}_\kk,\Z_2) &=& b^i(S'/\mathfrak{S}_\kk,\Z_2).
 \end{eqnarray}
 \end{enumerate}
 
 We now obtain an upper bound  $b(\widetilde{S}'_{\boldsigma},\Z_2)$ for each chain $\boldsigma \in \boldSigma_{\kk,\dd}$ as follows.
 Using Proposition \ref{prop:homotopic} we have that
 \[
 b(\widetilde{S}'_{\boldsigma},\Z_2) = b(\Ext(L_{\boldlambda},\R'_\ell)  \cap \widetilde{S}'_{\boldsigma},\Z_2),
 \]
 where $\boldlambda = \max(\boldsigma)$ and $\ell = \length(\boldlambda)$.

 Notice that $\widetilde{S}'_{\boldsigma}$ is the intersection of the $\mathcal{P}'$-closed semi-algebraic set $S$, with  the basic closed semi-algebraic set,
defined by 
 \begin{eqnarray}
 \nonumber
P_{\boldmu} - \eps_{\length(\boldmu)}  &=& 0, \mbox{ for $\boldmu \in \boldsigma, \boldmu \neq \boldlambda$}, \\
\label{eqn:number-of-distinct-general}
P_{\boldnu} - \eps_{\length(\boldnu)}  &\leq& 0, \boldnu \not\in \boldsigma, \boldnu \prec \boldlambda.
  \end{eqnarray}
 
 Let
 \[
 \mathcal{F}_{\boldsigma} = \bigcup_{\boldmu \in \boldsigma, \boldmu \neq \boldlambda} \{P_{\boldmu} - \eps_{\length(\boldmu)}\},
 \mathcal{G}_{\boldsigma} =  \bigcup_{\boldnu \not\in \boldsigma, \boldnu \prec \boldlambda}  \{P_{\boldnu} - \eps_{\length(\boldnu)} \}.
 \]
 
Using Corollary \ref{cor:chains}, the number of distinct subsets $\mathcal{G}_{\boldsigma}'\subset \mathcal{G}_{\boldsigma}$, such that 
 \[
 \ZZ(\mathcal{F}_{\boldsigma} \cup \mathcal{G}_{\boldsigma}',\R'_\ell) \cap \Ext(\Weyl^{(k)},\R'_\ell) \neq \emptyset
 \] 
 is bounded by 
 \begin{equation}
 \label{eqn:distinct-nonempty-general}
 \prod_{i=1}^{\omega}(O(d_i'))^{d_i'}.
 \end{equation}

We obtain using Proposition \ref{prop:P-closed-main} that
\begin{eqnarray*}
b(\widetilde{S}'_{\boldsigma},\F) &\leq & 
s+\sum_{ p \geq 0} \sum_{\substack{ I \subset [1,s],\\ 1 \leq \card(I) \leq k-p, \\
     J \subset I,\\ 1\leq \card(J) \leq p+1}} \sum_{\tau \in \{0,\pm 1,\pm 2\}^J} 
     \sum_{\mathcal{G}_{\boldsigma}'\subset \mathcal{G}_{\boldsigma}}  G(p,\card(I),J,K,\tau),
     \end{eqnarray*}
     where
    \[
     G(p,q,J,K,\tau)
     =
     b^{p+q -\card(J)} (\Ext(L_{\boldlambda},\R'_\ell) \cap \ZZ(\mathcal{P}_{\sigma} \cup \mathcal{F}_{\boldsigma} \cup \mathcal{G}_{\boldsigma}',\R'_\ell) \cap \Ext(\Weyl^{(k)},\R'_\ell),\F).
     \]

Since,
\[
\dim(L_{\boldlambda}) = \length(\boldlambda) \leq \sum_{i=1}^{\omega}d_i',
\] 
we obtain using
\eqref{eqn:distinct-nonempty}, Proposition \ref{prop:inductive},
and
Corollary \ref{cor:betti-bound-sa-general},
that,
 \begin{eqnarray}
  \label{eqn:bound:1-general} 
b(\widetilde{S}_{\boldsigma}',\Z_2)  &\leq &  \prod_{i=1}^\omega (O(\omega^3 s d_i d_i'))^{d_i'}.
  \end{eqnarray}
  The theorem now follows from 
  \eqref{eqn:thm:bound:S'-general}, 
 Propositions \ref{prop:chainsCompcat},  \ref{prop:MV}, and
  \eqref{eqn:bound:1-general}.
 \end{proof}

\subsection{Proofs of Theorems \ref{thm:definable1} and \ref{thm:definable2}}
\label{subsec:proofs-of-definable}
The proofs these theorems are  adaptations of the proofs of the corresponding theorems in the semi-algebraic case. These adaptations involve replacing infinitesimal elements by appropriately small positive elements of the ground field $\R$, and Hardt's triviality theorem for semi-algebraic sets by its
o-minimal version (see for example \S 5.7 \cite[Theorem 5.22]{Michel2}), similar to those already appearing in the proofs of the main results in \cite{Basu9}.  
The notion of $\lim_\eps S$, of a semi-algebraic set defined over $\R[\eps]$ which is bounded over $\R$, is replaced in the 
definable case by the intersection of the closure of the definable set $S' \subset \R^k \times \R$ 
with the hyperplane defined by $T=0$, where  $S'$ is the definable set obtained from $S$ by replacing $\eps$ by the new variable $T$.
If $S$ belongs to a definable family, the limit of $S$ defined this way would also belong to a definable depending only on the 
first definable family.

The final ingredient in the proofs of the bounds in the semi-algebraic case is
the use of Ole{\u\i}nik and Petrovski{\u\i} type bounds (cf. Theorem \ref{thm:betti-bound-algebraic}) 
to give a bound on the Betti numbers of semi-algebraic subsets of $\R^{d'}$, defined by polynomials of degree
at most $d$ (where $d' \leq d$) (cf. Corollary \ref{cor:betti-bound-sa-general}). In the definable case we will need to use the following
replacement of Corollary \ref{cor:betti-bound-sa-general}. 


\begin{proposition}
\label{prop:OPTM-definable}
\begin{enumerate}[1.]
\item
\label{itemlabel:prop:OPTM-definable:1}
Let $V \subset \R^m$ be a closed definable set in an o-minimal structure over $\R$ and $d > 0$. Then, there exists a constant $C_{V,d} > 0$ such that for all polynomial maps 
$F = (F_1,\ldots,F_{m}): \R^{d'} \rightarrow \R^m$, with $d' \leq d$ and $\deg(F_i) \leq d, 1 \leq i \leq m$, 
\[
b(f^{-1}(V), \F) \leq C_{V,d}.
\]
\item
\label{itemlabel:prop:OPTM-definable:2}
More generally, suppose that  $V \subset \R^m \times \R^\ell$ is a closed definable set in an o-minimal structure over $\R$, and $\pi_1: \R^m \times \R^\ell \rightarrow \R^m, \pi_2: \R^m \times \R^\ell \rightarrow \R^\ell$ be the two projection maps, and  for $\y \in \R^\ell$ denote by 
$V_\y$ the definable set $\pi_1(\pi_2^{-1}(y) \cap V)$. Then for each $d > 0$, there exists a constant $C_{V,d} > 0$, such that for every finite subset $A \subset \R^\ell$, and every $\mathcal{A}$-closed set $S \subset \R^m$, where $\mathcal{A} = \cup_{\y \in A} \{V_\y\}$, 
and all polynomial maps 
$F = (F_1,\ldots,F_{m}): \R^{d'} \rightarrow \R^m$, with $d' \leq d$ and $\deg(F_i) \leq d, 1 \leq i \leq m$, 
\[
b(F^{-1}(S), \F) \leq C_{V,d} \cdot n^{d'}.
\]
\end{enumerate}
\end{proposition}

\begin{proof}
Part\eqref{itemlabel:prop:OPTM-definable:1} of 
the proposition is a consequence of Hardt's triviality theorem for definable maps, which implies finiteness of topological types
amongst the definable sets $F^{-1}(V)$ as $F$ ranges all polynomial maps $F = (F_1,\ldots,F_m): \R^{d'} \rightarrow \R^m$, where 
the degree of each $F_i$ is at most $d$.

Part \eqref{itemlabel:prop:OPTM-definable:2} follows from Part \eqref{itemlabel:prop:OPTM-definable:1} and the proof of Theorem 2.3 in \cite{Basu9}.
\end{proof}

We sketch below the proofs of Theorem \ref{thm:definable1} and \ref{thm:definable2} indicating only the modifications needed
from the algebraic and semi-algebraic cases.

\begin{proof}[Sketch of proof of Theorem \ref{thm:definable1}]
The proof of Part \eqref{itemlabel:thm:definable1:1} is easy. In order to prove 
Part \eqref{itemlabel:thm:definable1:2} it suffices to modify appropriately the proof of Theorem \ref{thm:vanishing} replacing the symmetric semi-algebraic set $S$ with the symmetric definable set $S' = F^{-1}(V)$. Observe that the proof of Proposition \ref{prop:vanishing} remains valid if we replace the 
symmetric semi-algebraic set $S$ with $S'$ and ``semi-algebraic'' with ``definable'', after we observe that each polynomial $F_i$ is a polynomial in $p^{(k)}_1,\ldots,p^{(k)}_d$, and hence for each $\y \in \R^d$, 
$(\Psi^{(k)}_d)^{-1}(\y)$ maps on to a unique point in $\R^m$ under $F$, and
the fibre $(\Psi^{(k)}_d)^{-1}(\y)  \cap S'$ is either empty or equal to
$(\Psi^{(k)}_d)^{-1}(\y)$, depending on whether this point belongs to $V$ or not. 
Part \eqref{itemlabel:thm:definable1:2} now follows using the same argument as in the proof of 
Theorem \ref{thm:vanishing}.

In order to prove Part \eqref{itemlabel:thm:definable1:3}, observe again that the proof of Proposition \ref{prop:arnold} remains valid if we replace the 
symmetric semi-algebraic set $S$ with $S'$ and ``semi-algebraic'' with ``definable''. After replacing
the infinitesimals $\eps_i$ by appropriately small positive elements of $\R$,  
and $S$ by $S'$, definable analogs of Propositions 
\ref{prop:limit} (replacing the appropriately the notion of $\lim_\eps$ map by a definable analog as discussed above), 
\ref{prop:homotopic}, \ref{prop:MV} all hold. Finally, in order to prove 
Part \eqref{itemlabel:thm:definable1:3}, we replace the use of Corollary \ref{cor:betti-bound-sa-general} by Part \eqref{itemlabel:prop:OPTM-definable:1} of Proposition
\ref{prop:OPTM-definable}.
\end{proof}
   
\begin{proof}[Sketch of proof of Theorem \ref{thm:definable2}]
The proof is similar to that of proof of  Theorem \ref{thm:definable1}, except we replace the use of Corollary \ref{cor:betti-bound-sa-general} by Part \eqref{itemlabel:prop:OPTM-definable:2} of Proposition
\ref{prop:OPTM-definable} instead of Part \eqref{itemlabel:prop:OPTM-definable:1}. 
\end{proof}

\subsection{Proof of Theorem \ref{thm:algorithm}}
\label{subsec:algorithm}

Before proving Theorem \ref{thm:algorithm} we will need a few preliminary results that
we list below.

\subsubsection{Algorithmic Preliminaries}
\label{subsubsec:algo-prelim}
We begin with a notation.

\begin{notation}
Let $\mathcal{P} \subset \R [X_{1} , \ldots ,X_{k} ,Y_{1} , \ldots ,Y_{\ell}
]$ be finite, and let $\Pi$ denote a partition of the list of variables $X=
(X_{1} , \ldots ,X_{k} )$ into blocks, $X_{[1]} , \ldots ,X_{[ \omega ]}$,
where the block $X_{[i]}$ is of size $k_{i} ,1 \leq i \leq \omega$, $\sum_{1
\leq i \leq \omega} k_{i} =k$.

A $( \mathcal{P} , \Pi )$-formula $\Phi (Y)$ is a formula of the form
\[ \Phi (Y) = ( \Qu_{1} X_{[1]} ) \ldots ( \Qu_{\omega} X_{[ \omega ]} ) F
   (X,Y) , \]
where $\Qu_{i} \in \{\forall , \exists\}$, $Y= (Y_{1} , \ldots ,Y_{\ell} )$,
and $F (X,Y)$ is a quantifier free $\mathcal{P}$-formula.
\end{notation}

We will  use the following definition of complexity of algorithms in keeping with the convention used in
the book \cite{BPRbook2}.
\begin{definition}[Complexity of an algorithm]
\label{def:complexity}
By complexity of an algorithm that accepts as input a finite set of polynomials with coefficients in an ordered domain $\D$,
we will mean the number of ring operations (additions and multiplications) in $\D$, as well as the number of comparisons,
used in different steps of the algorithm. 
\end{definition}

The following algorithmic result on effective quantifier elimination is well-known.
We use the version stated in \cite{BPRbook2}.

\begin{theorem}\cite[Chapter 14]{BPRbook2}
\label{thm:qe}Let $\mathcal{P}$ be a
  set of at most $s$ polynomials each of degree at most $d$ in $k+ \ell$
  variables with coefficients in a real closed field $\R$, and let $\Pi$
  denote a partition of the list of variables~{$(X_{1} , \ldots ,X_{k} )$}
  into blocks, $X_{[1]} , \ldots ,X_{[ \omega ]}$, where the block $X_{[i]}$
  has size $k_{i} , {\mbox{for }  1 \leq i \leq \omega}$. Given $\Phi (Y)$, a
  $( \mathcal{P} , \Pi )$-formula, there exists an equivalent quantifier free
  formula,
  \[ \Psi (Y) = \bigvee_{i=1}^{I} \bigwedge_{j=1}^{J_{i}} (
     \bigvee_{n=1}^{Ni,j} \sign (P_{ijn} (Y))= \sigma_{ijn} ) , \]
  where $P_{ijn} (Y)$ are polynomials in the variables $Y$, $\sigma_{ijn} \in
  \{0,1, - 1\}$,
  \begin{eqnarray*}
    I & \leq & s^{(k_{\omega} +1) \cdots (k_{1} +1) ( \ell +1)} d^{O
    (k_{\omega} ) \cdots O (k_{1} ) O ( \ell )} ,\\
    J_{i} & \leq & 
      s^{(k_{\omega} +1) \cdots (k_{1} +1)} d^{O (k_{\omega} ) \cdots O (k_{1}
      )}
    ,\\
    N_{ij} & \leq &
      d^{O (k_{\omega} ) \cdots O (k_{1} )},
  \end{eqnarray*}
  and the degrees of the polynomials $P_{ijk} (y)$ are bounded by $d^{O
  (k_{\omega} ) \cdots O (k_{1} )}$. Moreover, there is an algorithm to
  compute $\Psi (Y)$ with complexity
\[
   s^{(k_{\omega} +1) \cdots (k_{1} +1) ( \ell +1)} d^{O (k_{\omega} )
     \cdots O (k_{1} ) O ( \ell )}.
\]
  \end{theorem}

\begin{corollary}
\label{cor:qe}
 There exists an algorithm that takes as input:
\begin{enumerate}[1.]
\item
$\mathcal{P}, \{F_1,\ldots,F_m\} \subset \D[\X]_{\leq d}$, where $\X=(X_1,\ldots,X_k)$;
\item
a $\mathcal{P}$-closed formula $\Phi$;  
\item
a set of linear $k-k'$ linear equations defining a linear subspace
$L \subset \R^k$
 of dimension $k'$;
 \end{enumerate}
 and computes a quantifier-free formula
 \[ \Psi (Y_1,\ldots,Y_m) = \bigvee_{i=1}^{I} \bigwedge_{j=1}^{J_{i}} (
     \bigvee_{n=1}^{Ni,j} \sign (P_{ijn} (\Y))= \sigma_{ijn} ) , \]
  where $P_{ijn} (\Y)$ are polynomials in the variables $\Y$, $\sigma_{ijn} \in
  \{0,1, - 1\}$,
  such that $\RR(\Psi,\R^m) = F(\RR(\Phi,\R^k) \cap L)$, 
  and $F = (F_1,\ldots,F_m): \R^k \rightarrow \R^m$ is the polynomial map defined by the tuple
  $(F_1,\ldots,F_m)$.
  
The complexity of the algorithm is bounded by 
\[ (s+m)^{(k' +1) (m +1)} d^{O (k') O (m)}, \]
where 
$s = \card(\mathcal{P})$. 

Moreover,
\begin{eqnarray*}
    I & \leq & s^{(k' +1) (m+1)} d^{O (k') O (m )} ,\\
    J_{i} & \leq & 
      (s+m)^{ (k' +1)} d^{O (k')},\\
    N_{ij} & \leq &
      d^{O (k')},
  \end{eqnarray*}
  and the degrees of the polynomials $P_{ijn}$ are bounded by $d^{O(k')}$.
\end{corollary}

\begin{proof}

First compute a basis of $L$, and replace $\mathcal{P}$ by
$\widetilde{\mathcal{P}} \subset \R[X_1',\ldots,X_{k'}']$ of pull-backs of polynomials in $\mathcal{P}$
to $L$, where $X_1',\ldots,X_{k'}'$ are coordinates with respect to the computed basis of $L$.
Similarly, replace the polynomials $F_1,\ldots,F_m$ by $\widetilde{F}_1,\ldots,\widetilde{F}_m$.
Replace the given formula $\Phi(X_1,\ldots,X_k)$ by a new formula $\widetilde{\Phi}(X_1',\ldots,X_{k'}')$ be replacing each occurrence of $P \in \mathcal{P}$ by the corresponding $\widetilde{P} \in \widetilde{\mathcal{P}}$.

Now apply Theorem \ref{thm:qe} with input the formula
\[
(\exists (X_1',\ldots,X_{k'}') \widetilde{\Phi} \wedge \bigwedge_{i=1}^{m} (Y_i - \widetilde{F}_i=0),
\]
to obtain the desired quantifier-free formula.

The complexity statement follows directly from that in Theorem \ref{thm:qe}.
\end{proof}

\begin{theorem}\cite{SS}
\label{thm:triangulation}
There exists an algorithm which takes as input a $\mathcal{P}$-closed formula defining a bounded semi-algebraic subset  $S$ of $\R^n$
with $\mathcal{P} \subset \D[X_1,\ldots,X_n]$, and
computes $b^i(S,\Q), 0 \leq i \leq n$. The complexity of this algorithm is bounded by $(\card(\mathcal{P}) D)^{2^{O(n)}}$, where
$D = \max_{P \in \mathcal{P}} \deg(P)$.
\end{theorem}

\begin{proof}
First compute a semi-algebraic triangulation of $h: |K| \rightarrow S$, 
where $K$ is a simplicial complex, $|K|$ the geometric realization of $K$, and $h$ s semi-algebraic homeomorphism,
as in the proof of Theorem 5.43 \cite{BPRbook2}. It is clear from the construction that the complexity, as
well as the size of the output,  is bounded by $(\card(\mathcal{P}) D)^{2^{O(n)}}$. Finally, compute the dimensions
of the simplicial homology groups of $K$ using for example the Gauss-Jordan elimination algorithm from elementary linear algebra.
Clearly, the complexity remains bounded by $(\card(\mathcal{P} )D)^{2^{O(n)}}$.
\end{proof}

\subsubsection{Proof of Theorem \ref{thm:algorithm}}
\label{subsubsec:proof-of-thm:algorithm}
We are finally in a position to prove Theorem \ref{thm:algorithm}.
\begin{proof}[Proof of Theorem \ref{thm:algorithm}]
We first prove using Corollary \ref{cor:qe} that it is possible to compute a quantifier-free
$\Theta$ such that
$\RR(\Theta,\R^d) = \Psi^{(k)}_{d}(S)$,
and the complexity of this step being bounded by
\[
k^{O(d)} (s d)^{O(d^2)}.
\]

To see this apply for each $\lambda \in \Comp(k,d)$ with $\length(\lambda) = d$, apply Corollary \ref{cor:qe} to obtain a
formula $\Theta_\lambda$ such that \[
\RR(\Theta_\lambda,\R^d) = \Psi^{(k)}_d(S \cap \Weyl_\lambda).
\]

The complexity of this step using the complexity statement in Corollary \ref{cor:qe} is bounded by 
$(s d)^{O(d^2)}$, noting that $\Weyl_\lambda \subset L_\lambda$ and $\dim L_\lambda \leq d$. Moreover,
the same bound applies to the number and the degrees of the polynomials appearing in $\Theta_\lambda$.

 Finally, we can take 
 \[
  \Theta = \bigvee_{\substack{\lambda \in \Comp(k,d), \\ \length(\lambda) = d}} \Theta_\lambda.
  \]
  
  Note that 
  
  \begin{eqnarray}
  \label{eqn:thm:algorithm:1}
  \card(\Comp(k,d)) &\leq& O(k)^d
  \end{eqnarray}
 (cf. Proposition \ref{prop:chainsCompcat}). 
  
Finally, we compute the Betti numbers of $\Psi^{(k)}_d(S) = \RR(\Theta,\R^d)$, using
Theorem \ref{thm:triangulation}.
Using the complexity of the algorithm in Theorem \ref{thm:triangulation}, and \eqref{eqn:thm:algorithm:1},
we see that the complexity of this step is bounded by
\begin{eqnarray*}
\left((O(k))^d (s d)^{O(d^2)}\right)^{2^{O(d)}} &=& (s k d)^{2^{O(d)}}.
\end{eqnarray*}

Finally, using Proposition \ref{prop:vanishing} we have that,
\[
b^i(S/\mathfrak{S}_k,\F) = b^i(\Psi^{(k)}_d(S),\F), 0 \leq i < d,
\]
finishing the proof.
\end{proof}

\section{Conclusion and Open Problems}\label{sec:conclusion}
In this paper we have improved on previous bounds on equivariant Betti numbers for symmetric semi-algebraic sets. It would be interesting to extend the method used in this paper to other situations. Currently, it seems that Kostov's result which was a central ingredient of the approach used here relies on a particular choice of generators for the ring of symmetric polynomials. Therefore, it is up to further investigation if a similar result holds for other groups acting on the ring of polynomials.

On the algorithmic side, we showed that it is possible to design an efficient algorithm to compute the equivariant Betti numbers. It has been shown in  \cite{BR2014} that not only the equivariant Betti numbers can be bounded polynomially, but in fact that the multiplicities of the various irreducible representations occurring in an isotypic decomposition of the homology groups of symmetric semi-algebraic sets can also be bounded polynomially.  Building on this result it is an interesting question to ask if an algorithm with similar complexity can be designed to compute these multiplicities as well,  and thus in fact computing all the Betti numbers  of symmetric varieties with complexity that is polynomial in $k$, for every fixed $d$.
\section*{Acknowledgment}
The research presented in this article  was initiated during a stay of the authors at  Fields Institute as part of the Thematic Program on Computer Algebra and the authors would like to thank the organizers of this event. 

\bibliographystyle{amsplain}
\bibliography{master}
\end{document}